\documentclass[12pt,reqno,a4paper]{amsart}
\usepackage{xcolor}
\usepackage{amsmath}
\usepackage{amsthm}
\usepackage{bbm}
\usepackage{amssymb}
\usepackage{mathrsfs}
\usepackage{graphicx}
\usepackage{tikz}
\usepackage{bm}
\usepackage{ucs}
\usepackage{wrapfig}
\usepackage{color}
\usepackage{float}
\usepackage{fullpage}
\usepackage{epstopdf}
\usepackage[breaklinks,colorlinks=true,linkcolor=blue,citecolor=red,
  backref=page]{hyperref}

\newtheorem{remark}{Remark}

\newtheorem{prop}{Proposition}
\newtheorem{lemma}{Lemma}
\newtheorem{theorem}{Theorem}
\newtheorem{corollary}{Corollary}

\newcommand{\ind}{\,\mbox{d}}
\newcommand{\normmm}[1]{{\left\vert\kern-0.3ex\left\vert\kern-0.25ex\left\vert #1 
    \right\vert\kern-0.25ex\right\vert\kern-0.3ex\right\vert}}

\begin{document}
\date{}

\title{Data-driven basis for reconstructing the contrast in inverse  scattering: Picard criterion, regularity, regularization, and stability}
\author[]{Shixu Meng$^1$}\footnotetext[1]{Academy of Mathematics and Systems Science, Chinese Academy of Sciences,
Beijing 100190, China.  {\tt shixumeng@amss.ac.cn}}
\begin{abstract}
We consider the inverse medium scattering of reconstructing the medium contrast using Born data, including the full aperture, limited-aperture, and multi-frequency data. We propose data-driven basis functions for these inverse problems based on the generalized prolate spheroidal wave functions and related eigenfunctions. Such data-driven eigenfunctions are eigenfunctions of a Fourier integral operator;  they remarkably extend analytically to the whole space, are  doubly orthogonal, and are complete in the class of band-limited functions. We first establish a Picard criterion for reconstructing the contrast using the data-driven basis, where the reconstruction formula can also be understood from the viewpoint of data processing and analytic extrapolation. Another salient feature associated with the generalized prolate spheroidal wave functions  is that the data-driven basis for a disk is also a basis for a Sturm-Liouville differential operator.  With the help of  Sturm-Liouville theory, we estimate the $L^2$ approximation error for a spectral cutoff approximation of $H^s$ functions. This yields a spectral cutoff regularization strategy for noisy data and an explicit stability estimate for contrast in $H^s$ { ($0<s<1/2$)} in the full aperture case. In  the limited-aperture  and multi-frequency cases, we also obtain spectral cutoff regularization strategies for noisy data and  stability estimates for a class of contrast.  
\end{abstract} 

\maketitle

\textbf{Key Words.} inverse medium scattering,  generalized prolate spheroidal wave function,  Born approximation, Picard criterion, regularization, stability.

\section{Introduction} \label{section introduction}
Inverse scattering is of great importance in non-destructive testing, medical imaging, geophysical exploration, and numerous problems associated with target identification. 
Let us first introduce the  inverse medium scattering problem in two dimensions.   Let $k>0$ be the  wave number. A plane wave  takes the following form:
\begin{equation*} 
e^{i k x \cdot \hat{\theta}}, \quad \hat{\theta} \in \mathbb{S} :=\{ x \in \mathbb{R}^2: |x|=1\},
\end{equation*}
where $\hat{\theta}$ is the direction of propagation.
Let $\Omega \subset \mathbb{R}^2$ be an open and bounded set with
Lipschitz boundary $\partial \Omega$ such that $\mathbb{R}^2 \backslash \overline{\Omega}$  is connected. The set $\Omega$ is referred to as the medium. Let the real-valued function $q(x) \in L^\infty(\Omega)$ be the contrast of the medium supported in $\Omega$ (which gives rise to the refractive index  { ${1+q}$}) and $q \ge 0$ on $\Omega$. The medium scattering due to a plane wave $e^{i k x \cdot \hat{\theta} }$ is to find total wave field $e^{i k x \cdot \hat{\theta} } + u^s(x;\hat{\theta};k)$ belonging to $H^1_{loc}(\mathbb{R}^2)$  such that
\begin{eqnarray}
\Delta_x \big( u^s(x;\hat{\theta};k) + e^{ikx\cdot \hat{\theta}} \big) + k^2 \left(1+q(x)\right) \big( u^s(x;\hat{\theta};k) + e^{ikx\cdot \hat{\theta}} \big) =  0 \quad &\mbox{in}& \quad \mathbb{R}^2, \quad  \label{medium us eqn1}\\
\lim_{r:=|x|\to \infty} \sqrt{r}  \big( \frac{\partial u^s(x;\hat{\theta};k)}{\partial r} -ik u^s(x;\hat{\theta};k)\big) =0,  \label{medium us eqn2}
\end{eqnarray}
{where \eqref{medium us eqn2} holds  uniformly in all directions}. The scattered wave field is $u^s(\cdot;\hat{\theta};k)$. A solution is called   radiating   if it satisfies \eqref{medium us eqn2}.   This scattering problem is well-posed and there exists a unique radiating solution to \eqref{medium us eqn1}--\eqref{medium us eqn2}; cf. \cite{colton2012inverse,kirsch2008factorization}. We refer to \eqref{medium us eqn1}--\eqref{medium us eqn2} as the full   model.
 
 Born approximation is a widely used  method to treat inverse problems; cf. \cite{colton2012inverse,moskowSchotland08,serov07}.
 In the Born approximation region, one can approximate the solution $u^s(\cdot;\hat{\theta};k)$ by its Born approximation $u^s_b(\cdot;\hat{\theta};k)$ \cite{colton2012inverse}, which is the unique radiating solution to
\begin{eqnarray*}
\Delta_x u^s_b(x;\hat{\theta};k) + k^2  u^s_b(x;\hat{\theta};k) = -k^2 q(x) e^{ikx \cdot \hat{\theta}} \quad &\mbox{in}& \quad \mathbb{R}^2. \quad  \label{medium us born eqn1}
\end{eqnarray*}
Every radiating solution of the Helmholtz equation has the following asymptotic
behavior at infinity \cite{cakoni2016qualitative}:
\begin{equation*} 
u_b^{s}(x;\hat{\theta};k)
=\frac{e^{i\frac{\pi}{4}}}{\sqrt{8k\pi}} \frac{e^{ikr}}{\sqrt{r}}\left\{\widetilde{u}_b^{\infty}(\hat{x};\hat{\theta};k)+\mathcal{O}\left(\frac{1}{r}\right)\right\}\quad\mbox{as }\,r=|x|\rightarrow\infty,
\end{equation*}
uniformly with respect to all directions $\hat{x}:=x/|x|\in\mathbb{S}$. The complex-valued function $\widetilde{u}_b^{\infty}(\hat{x};\hat{\theta};k)$ defined on $\mathbb{S}$ is known as the scattering amplitude or  far-field pattern with $\hat{x}\in\mathbb{S}$ denoting the observation direction. It directly follows from \cite{cakoni2016qualitative} that
\begin{eqnarray} \label{born uinfty integral representation}
\widetilde{u}_b^{\infty}(\hat{x};\hat{\theta};k) = k^2 \int_\Omega  e^{-ik \hat{x} \cdot p'} q(p') e^{ik p' \cdot \hat{\theta}} \ind p'.
\end{eqnarray}

The linear sampling method \cite{ColtonKirsch96} and factorization method \cite{Kirsch98} developed in the 1990s apply to the full model of the inverse medium scattering and  apply to the Born model \eqref{born uinfty integral representation} more explicitly (see \cite{Kirsch17} for a comparison between the full far-field operator and the Born far-field operator). Using a far-field operator $L^2(\mathbb{S}) \to L^2(\mathbb{S})$ whose kernel is the far-field pattern, these sampling methods are capable of reconstructing the support of the contrast $q$, which requires little a priori information about the scattering objects, and provides both theoretical justifications and robust numerical algorithms. In particular, the linear sampling method \cite{ColtonKirsch96} and the factorization method \cite{Kirsch98} as well as the convex scattering support \cite{KusiakSylvester05,SylvesterKelly}  use an infinite series and Picard criterion for the reconstruction. We refer the reader to \cite{cakoni2016inverse,kirsch2008factorization} for a more comprehensive introduction. 

In this paper we follow the same idea and aim to reconstruct not only the shape but also the contrast by  \textit{data-driven basis functions} related to the \textit{generalized  prolate spheroidal wave functions} \cite{Slepian64}. The idea begins with formulating the inverse problem as follows (see also \cite[Section 7.4]{Kirsch21}).

\textbf{Inverse problem} with full aperture Born data: determine the contrast $q \in L^2(\Omega)$ from 
\begin{equation*}
\{{ u_{b}^{\infty}}(p;k): p \in B(0,2)\}
\end{equation*}
where ${ u_{b}^{\infty}}(p;k)$ is given by \eqref{FM near field case 2layer filtered data def 1 born volume eqv} 
\begin{eqnarray} \label{FM near field case 2layer filtered data def 1 born volume eqv}
{ u_{b}^{\infty}}(p;k)  := \int_\Omega e^{i k p\cdot p'} q(p')   \ind p'.
\end{eqnarray}
This equivalent formulation is due to \eqref{born uinfty integral representation} and the fact that $B(0,2)$ is the interior of  $\{ \hat{\theta}-\hat{x}:\hat{x}, \hat{\theta} \in \mathbb{S}\}$. Similar to the linear sampling/factorization method where one investigates the far-field operator $L^2(\mathbb{S}) \to L^2(\mathbb{S})$, by processing the data as \eqref{FM near field case 2layer filtered data def 1 born volume eqv} we consider an operator $L^2(\Omega) \to L^2(B(0,2))$  (which will be formulated shortly as an operator $L^2(D_F) \to L^2(D_F)$ for some set $D_F$) to treat the inverse problem.  Note that by formulating the Born inverse scattering problem as \eqref{FM near field case 2layer filtered data def 1 born volume eqv}, the comparison  between the full far-field operator and the operator associated with \eqref{FM near field case 2layer filtered data def 1 born volume eqv} may be different from \cite{Kirsch17}.

The inverse problems in the form of \eqref{FM near field case 2layer filtered data def 1 born volume eqv} is important in science, engineering, and technology. Its application is beyond Born inverse scattering:  it also merits application in the inverse source problem (in which case $q$  stands for the intensity of the source, $p$  stands for the observation direction, and $k$ still stands for the frequency/wave number), computerized tomography \cite{natterer2001book}, and Fourier analysis \cite{Slepian64,Slepian78,Slepian61,ZLWZ20}.  We are  most interested in inverse scattering and thereby have chosen to introduce the problem in the context of inverse scattering. 

Our data-driven basis in this paper yields highly accurate and efficient reconstruction algorithms. Such a basis stems from an eigensystem $\{\psi_{m,n,\ell}(\cdot;c), \alpha_{m,n}(c)  \}_{m,n \in \mathbb{N}}^{\ell  \in \mathbbm{I}(m)}$    that satisfies
\begin{eqnarray*} 
 \int_{B(0,1)} e^{i c p \cdot p' } \psi_{m,n,\ell}(p';c) \ind p'  
= \alpha_{m,n}(c)  \psi_{m,n,\ell}(p;c), \qquad p \in B(0,1),
\end{eqnarray*}
where $c>0$ is a positive constant, $\mathbb{N}=\{0,1,2,\cdots\}$, $\mathbbm{I}(m)=\{1\}$ if $m=0$, and $\mathbbm{I}(m)=\{1,2\}$ if $m\ge1$. Orthogonal basis functions serve as an important tool in analyzing inverse problems and in establishing regularization strategies \cite[Chapters 2--3]{Kirsch21}. 
This data-driven basis  originated in the work of Slepian \cite{Slepian64} who gave a thorough theory on  the generalized prolate spheroidal wave function (which is the radial part of $\psi_{m,n,\ell}(\cdot;c)$) and the related Fourier analysis in multiple dimensions. In one dimension, the generalized  prolate spheroidal wave function corresponds to the prolate spheroidal wave function \cite{Slepian61}. Remarkably, the (generalized) prolate spheroidal wave function is the eigenfunction of both a Fourier type integral operator and a Sturm-Liouville differential operator, which merits important applications in analytic extrapolation, approximation theory, uncertainty quantification, and Fourier analysis \cite{Slepian64,Slepian78,Slepian61,ZLWZ20}. Though such a data-driven basis applies naturally to \eqref{FM near field case 2layer filtered data def 1 born volume eqv}, the application of generalized prolate spheroidal wave functions is rather limited in multi-dimensional inverse scattering and inverse source problems to the best of our knowledge. Indeed, Slepian's discrete prolate spheroidal wave function (in one dimension) \cite{Slepian78} applied naturally to the limited angle problem of computerized tomography \cite[Chapter VI]{natterer2001book}. It amounts to extrapolating the limited angle data to obtain the full angle data. Recently, Slepian's discrete prolate spheroidal wave function was applied in limited-aperture inverse scattering in $\mathbb{R}^2$ \cite{DLMZ2021}. In particular, the discrete prolate spheroidal wave functions are  doubly orthogonal (over the limited-aperture and the full aperture) and they serve as a data-driven basis for data completion or a Galerkin projection basis for the limited-aperture factorization method \cite[Corollary 2.17]{kirsch2008factorization}.  It would be desirable to develop an explicit stability estimate for the data completion algorithm of \cite{DLMZ2021} in the spirit of this work. Note that the discrete prolate spheroidal wave functions can be derived by a high-fidelity trigonometric approximation of the data followed by applying a reduced basis method with singular value decomposition \cite{DLMZ2021}; this is in line with the reduced basis method in \cite{quarteroni2016book}. As such we are motivated to call the set of prolate related functions  a data-driven basis. { In a broader context, it is possible to apply a reduced basis method  such as \cite{quarteroni2016book} or  other physics-informed basis functions using machine learning such as \cite{Karniadakis21} to study the inverse problems. The data-driven basis in this paper can also be learned via a Legendre-Galerkin neural network; in this paper we study this data-driven basis analytically  by the generalized prolate  spheroidal wave functions 
and remark that our result indeed serves as a mathematical foundation for relevant machine learning algorithms.} We also note when our work is near completion that \cite{novikov22} applied the one dimensional prolate spheroidal wave functions together with the Radon transform to study reconstructions from the Fourier transform on the ball.

In this work we show how to use  the generalized prolate spheroidal wave functions and their related eigenfunctions to solve the Born inverse scattering problem. We emphasize the following: (1) First, the generalized prolate spheroidal wave functions and their related eigenfunctions \cite{Slepian64} form a data-driven basis for a Fourier integral operator associated with \eqref{FM near field case 2layer filtered data def 1 born volume eqv}. This allows us to establish a Picard criterion to reconstruct the contrast. This is in the spirit of the factorization/linear sampling method. (2) Second, note that if one has the knowledge of the data in the whole space $\mathbb{R}^2$ (which is guaranteed by unique continuation in theory), then one can solve for the contrast by classical inverse Fourier transform. The data-driven basis remarkably extends analytically to $\mathbb{R}^2$, is  doubly orthogonal, and is complete in the class of band-limited functions \cite{Slepian64,Slepian61}. We show that  the reconstruction by Picard criterion can be understood from the viewpoint of data processing where one extrapolates the data to $\mathbb{R}^2$ first and then applies the inverse Fourier transform. The ill-posed nature of analytic extrapolation  is revealed by the decay of the eigenvalues associated with the data-driven basis. (3) Third, the generalized prolate spheroidal wave functions (i.e., radial part of the data-driven basis) is also a basis for a Sturm-Liouville differential operator in the radial variable \cite{Slepian64}, and remarkably the data-driven basis (in the two-dimensional variable $x$) is indeed a basis for another self-adjoint, positive definite Sturm-Liouville differential operator in the variable $x$. Thereby one can study the regularity estimate using such a Sturm-Liouville operator. In particular, we show that the data-driven basis yields a natural definition of a suitable Sobolev space $\widetilde{H}^s_c$ in connection with the classical Sobolev space $H^s$ and we estimate the $L^2$ approximation error for a spectral cutoff approximation of functions in $H^s$, $0<s\le 1$. This  type of estimate is of independent interest in approximation theory using the data-driven basis. (4) Last but not least, as the data is ``on the right hand side'' of the operator equation, by the regularity estimate we are able to obtain an explicit regularization strategy for noisy data and establish a stability estimate for contrast (or more precisely, an extension of the contrast) belonging to the classical Sobolev space $H^s$, $0<s<1/2$. {  Such a stability estimate is expected to give a possible H\"{o}lder-logarithmic type stability estimate by choosing appropriate parameters in the regularization.}

Limited-aperture data (cf. \cite{Haddar_etal17,BaoLiu,harrisRezac22,IkehataNiemiSiltanen,LiLiuZou14}) and multi-frequency partial data (cf. \cite{AlaHuLiuSun,BaoLiLinTriki,GriesmaierSchmiedecke-source,LiuMeng2021,meng22}) pose great challenges in inverse scattering. These partial data indicate that inverse problems are more ill-posed. In this work, in addition to the full aperture case, we  moreover study  the following limited-aperture   and   multi-frequency inverse problems.

\textbf{Inverse problem} with limited-aperture   and   multi-frequency  data: determine the contrast $q \in L^2(\Omega)$ from
\begin{itemize}
\item[-] limited-aperture Born data $\{ \widetilde{u}_b^{\infty}(\hat{x};\hat{\theta};k): \hat{x}, \hat{\theta} \in \mathbb{S}_L\}$, $\mathbb{S}_L:=\{x: x \in \mathbb{S}, \arg x \in [-\Theta,\Theta],  0<\Theta<\pi\}$,
where these data are equivalent to
\begin{equation*}
 \{{ u_{b}^{\infty}}(p;k): p \in L\};
\end{equation*}
here $L$ is the interior of $\{\hat{\theta}-\hat{x}: \hat{x}, \hat{\theta} \in \mathbb{S}_L\}$ which is symmetric with respect to the origin, and ${ u_{b}^{\infty}}$  is given by \eqref{FM near field case 2layer filtered data def 1 born volume eqv}.
\item[-] multi-frequency Born data with two opposite observation directions $\{\widetilde{u}_b^{\infty}(\pm \hat{x}^*;  \hat{\theta};k): \hat{\theta} \in \mathbb{S}, k \in (0,K), K>0\}$
which are equivalent to
\begin{equation*}
 \{{ u_{b}^{\infty}}(p;K): p \in M\};
\end{equation*}
where $M$ is the interior of $\{a\hat{\theta} \pm a\hat{x}^*: \hat{\theta} \in \mathbb{S}, a \in (0,1)\}$ which is symmetric with respect to the origin, and ${ u_{b}^{\infty}}$ is given by \eqref{FM near field case 2layer filtered data def 1 born volume eqv}.
\end{itemize}
In each of the limited-aperture  and multi-frequency cases, the inverse problem corresponds to inverting the integral operator with data in a   symmetric set $L$ and $M$, respectively (as opposed to the full aperture case where we invert the integral operator with data in a disk $B(0,2)$).

The application of the generalized prolate spheroidal wave functions (and their related functions) is far less developed for the limited-aperture and multi-frequency inverse scattering problems. Remarkably, there also exists a  data-driven basis for a general symmetric set \cite{Slepian64}. The above Picard criterion and analytic extrapolation apply in exactly the same way to the limited-aperture and multi-frequency cases. Due to the fact that  the area of   the symmetric set is in general  less than the area of $B(0,2)$ in the full aperture case, the Picard criterion and analytic extrapolation expose more ill-posedness (see \cite{DLMZ2021,LMZ22,Tsogka_2016} for similar observations).
Due to the lack of a Sturm-Liouville theory for the general symmetric set (to the best of our knowledge),  the stability estimate differs from the one in full aperture case, where the stability estimate for partial data relies on less explicit a priori information on the contrast (for instance, we require that  the contrast belongs to the range of some operator). 

The remainder of the paper is organized as follows. In Section \ref{section GPSWF}, we summarize the  generalized prolate spheroidal wave functions and Fourier analysis \cite{Slepian64} that are needed in our study. In Section \ref{Picard Criterion and analytic extrapolation} we  generate the data-driven basis for the full aperture data and establish a Picard criterion to reconstruct the contrast  in the spirit of the factorization/linear sampling method. We also show that the reconstruction by the Picard criterion can be understood from the viewpoint of data processing where one extrapolates the data to $\mathbb{R}^2$ first and then applies the inverse Fourier transform. Section \ref{section full aperture data regularity regularization stability} is devoted to  the regularity estimate, regularization strategy, and stability estimate. In particular, with the help of the Sturm-Liouville theory developed in Section \ref{section SLP}, we estimate in Section \ref{section approximation in L2} the $L^2(B(0,1))$ approximation error for a  spectral cutoff approximation of functions in $ H^s(B(0,1))$, $0<s \le 1$. This allows us to develop the spectral cutoff regularization and its stability estimate in Section \ref{section full aperture regularization stabiltiy}. Finally we develop in Section \ref{section limited-aperture multi-frequency data} the Picard criterion and stability estimate for the limited-aperture  and multi-frequency cases.

 Throughout the paper,  we consistently use the superscript/subscript ``F'',``L'', and``M'' to represent that relevant functions, constants, and operators are for the full aperture data, the limited-aperture data, and the multi-frequency partial data, respectively.

\section{Generalized prolate spheroidal wave functions and Fourier Analysis} \label{section GPSWF}
In this section, we summarize some of the main results \cite{Slepian64} that are needed in our study. Let $c$ be a positive constant. We are interested in finding the eigensystem of the integral operator  $\mathcal{F}_A^c: L^2(A) \to L^2(A)$  
\begin{eqnarray*}
(\mathcal{F}_A^c\psi)(p) = \int_{A} e^{i c p \cdot p'} \psi(p') \ind p', \quad \forall p \in A, \forall \psi \in  L^2(A)
\end{eqnarray*}
where $A$ is a symmetric ($x \in A$ if and only if $-x \in A$), bounded, open set. We summarize below a particular theory for a unit disk $B(0,1)$ and a general theory for a general symmetric set $A$. All the following results in this section are from \cite{Slepian64}.
\subsection{Generalized prolate spheroidal wave functions for unit disk} \label{section GPSWF disk}
When the set $A$ is a unit disk $B(0,1)$, there exists an orthogonal eigensystem $\{\psi_{m,n,\ell}(\cdot;c), \alpha_{m,n}(c)  \}_{m,n \in \mathbb{N}}^{\ell  \in \mathbbm{I}(m)}$    \cite[pp. 3015--3018]{Slepian64} that satisfies
\begin{eqnarray}\label{prop operator eigenfunction gpswf B01}
(\mathcal{F}_{B(0,1)}^c\psi_{m,n,\ell}(\cdot;c))(p) &=& \int_{B(0,1)} e^{i c p \cdot p' } \psi_{m,n,\ell}(p';c) \ind p' \nonumber \\
&=& \alpha_{m,n}(c)  \psi_{m,n,\ell}(p;c), \qquad p \in B(0,1)
\end{eqnarray}
with $\mathbb{N}=\{0,1,2,\cdots\}$ and
\begin{eqnarray*}
\mathbbm{I}(m) =
\left\{
\begin{array}{cc}
\{1\}  & m=0     \\
\{1,2\}  &     m\ge 1 
\end{array}
\right.  .
\end{eqnarray*}

\begin{itemize}
\item Each eigenfunction $\psi_{m,n,\ell}(\cdot;c)$ can be obtained via  separation of variables
\begin{eqnarray} \label{section GPSWF disk eqn psi and varphi}
\psi_{m,n,\ell}(\cdot;c)(x) = |x|^{-\frac{1}{2}}\varphi_{m,n}(|x|;c) Y_{m,\ell} (\theta_{x}), \quad x \in B(0,1), \quad \theta_{x} = \arg x,
\end{eqnarray}
where  
\begin{eqnarray*}
Y_{m,\ell} (\theta) =
\left\{
\begin{array}{ccc}
1  & m=0,\ell=1     \\
\cos(m\theta)  &  m\ge 1, \ell=1    \\
\sin(m\theta)  &     m\ge 1, \ell=2
\end{array}
\right.  
\end{eqnarray*}
and { $\{\varphi_{m,n}(\cdot;c), \gamma_{m,n}(c)  \}_{m,n \in \mathbb{N}}$} is the eigensystem that satisfies
\begin{eqnarray} \label{GPSWF def B01 eig kernel J}
\int_0^1 J_m(c r r') \sqrt{crr'} \varphi_{m,n}(r';c) \ind r' = \gamma_{m,n}(c)  \varphi_{m,n}(r;c), \quad 0<r<1,
\end{eqnarray}
here the eigenvalue $\gamma_{m,n}(c)$ is related to $\alpha_{m,n}(c)$ by  $\gamma_{m,n}(c)=\frac{c^{1/2}}{2\pi i^m}\alpha_{m,n}(c)$. The function $\varphi_{m,n}(\cdot;c)$ is chosen as real-valued and it is called a \textit{generalized prolate spheroidal wave function} according to Slepian.
\item Every eigenvalue $\alpha_{m,n}(c)$ is non-zero and
$$
|\alpha_{m,n}(c)| \to 0, \quad m,n\to \infty.
$$
 $\{\psi_{m,n,\ell}(\cdot;c) \}_{m,n \in \mathbb{N}}^{\ell  \in \mathbbm{I}(m)}$ is a  complete and orthogonal set in $L^2(B(0,1))$. 
\item By extending the domain of each $\psi_{m,n,\ell}(\cdot;c)$ to $\mathbb{R}^2$   naturally (and without the danger of confusion) via 
\begin{eqnarray} \label{prop operator eigenfunction gpswf B01 extension}
\psi_{m,n,\ell}(p;c) := \frac{1}{ \alpha_{m,n}(c)}\int_{B(0,1)} e^{i c p \cdot p' } \psi_{m,n,\ell}(p';c) \ind p' , \quad p \in \mathbb{R}^2,
\end{eqnarray}
one can get the double orthogonality \cite[pp. 3013--3015]{Slepian64}, for any $m,n\in \mathbb{N}, \ell  \in \mathbbm{I}(m)$,
\begin{eqnarray} 
&&\int_{B(0,1)} \psi_{m,n,\ell}(p;c) \psi_{m',n',\ell'}(p;c) \ind p \label{GPSWF def B01 psi double orthogonal eqn 1} \nonumber \\
&&=  \left(\frac{c}{2\pi} \right)^2  |\alpha_{m,n}(c)|^2   \int_{\mathbb{R}^2} \psi_{m,n,\ell}(p;c) \psi_{m',n',\ell'}(p;c) \ind p \label{GPSWF def B01 psi double orthogonal eqn 2}
\end{eqnarray}
 and each  $\psi_{m,n,\ell}(\cdot;c)$ is normalized such that its energy in $\mathbb{R}^2$ is unit or equivalently
\begin{eqnarray*}
\int_{B(0,1)}  \psi_{m,n,\ell}(p;c)\psi_{m',n',\ell'}(p;c)    \ind p &=& \left(\frac{c}{2\pi}\right)^2  |\alpha_{m,n}(c)|^2 \delta_{mm'}\delta_{nn'}\delta_{\ell \ell'},
\end{eqnarray*}
here $\delta$ denotes the Kronecker delta. For the purpose of a clean presentation later on, we use $\|\psi_{m,n,\ell}(\cdot;c)\|$ to represent the $\|\psi_{m,n,\ell}(\cdot;c)\|_{L^2(B(0,1))}$ norm of the eigenfunction $\psi_{m,n,\ell}(\cdot;c)$ (and similar notation holds for other doubly orthogonal eigenfunctions); when we discuss the $\|\psi_{m,n,\ell}(\cdot;c)\|_{L^2(\mathbb{R}^2)}$ norm, we shall state this explicitly.
\item
The remarkable property of the generalized prolate spheroidal wave function (eigenfunctions given by \eqref{GPSWF def B01 eig kernel J}) is that $\varphi_{m,n}(\cdot;c)$  is also the eigenfunction of   the  (singular) Sturm–Liouville differential operator 
\begin{eqnarray} 
&& \mathcal{D}_{c,r} \varphi_{m,n}(\cdot;c) = \chi^o_{m,n}(c), \quad \mbox{ where }  \label{GPSWF def B01 SLP Dcr eqn 1} \\
&& \mathcal{D}_{c,r}:=-(1-r^2) \frac{\ind^2 }{ \ind r^2} + 2r \frac{\ind   }{ \ind r} - \Big(  \frac{1/4-m^2}{r^2} -c^2  r^2 \Big), \label{GPSWF def B01 SLP Dcr eqn 2}
\end{eqnarray}
and $ \chi^o_{m,n}(c)$ is the corresponding  eigenvalue.
When $m=1/2$, this equation reduces to the equation for prolate spheroidal wave functions of order zero \cite{Slepian61}.
\end{itemize}

Remarkably, the (generalized) prolate spheroidal wave function  is the eigenfunction of both a Fourier type integral operator and a Sturm-Liouville differential operator, which merits important applications in analytic extrapolation, approximation theory, uncertainty quantification, and Fourier analysis \cite{Slepian64,Slepian61}.

\subsection{Eigensystem for symmetric set}  \label{section GPSWF symmetric}
Recall that $A \subset \mathbb{R}^2$ is a  symmetric  bounded open set.
Consider the eigenvalue problem of finding eigenfunctions $\psi_e(\cdot;c),\psi_o(\cdot;c) \in L^2(A)$ and  eigenvalues $\beta_e(c), \beta_o(c)$ such that
\begin{eqnarray*}
\beta_e(c) \psi_e(p;c) &=& \int_{A} \cos(c p \cdot p') \psi_e(p';c) \ind p',\quad p \in A, \\
\beta_o(c) \psi_o(p;c) &=& \int_{A} \sin(c p \cdot p') \psi_o(p';c) \ind p',\quad p \in A.
\end{eqnarray*}
Since $\cos$ is an even function and $A$ is symmetric, it follows from  \cite[pp. 3014--3015]{Slepian64} that there exists an eigensystem $\{\psi_{e,n}(\cdot;c), \beta_{e,n}(c)\}_{n=0}^\infty$ such that $\{\psi_{e,n}(\cdot;c)\}_{n=0}^\infty$ is complete in the set of even functions in $L^2(A)$, and  $\psi_{e,n}(\cdot;c)$ and  $\beta_{e,n}(c)$ are real-valued. Similarly  there exists an eigensystem $\{\psi_{o,n}(\cdot;c), \beta_{o,n}(c)\}_{n=0}^\infty$ such that $\{\psi_{o,n}(\cdot;c)\}_{n=0}^\infty$ is complete in the set of odd functions in $L^2(A)$, and  $\psi_{o,n}(\cdot;c)$ and  $\beta_{o,n}(c)$ are real-valued. The set $\{\psi_{o,n}(\cdot;c)\}_{n=0}^\infty \cup \{\psi_{e,n}(\cdot;c)\}_{n=0}^\infty$ is complete in $L^2(A)$.

In this way one can obtain the eigensystem  $\{\psi_{n}(\cdot;c), \alpha_{n}(c)\}_{n=0}^\infty$ := $\{\psi_{e,n}(\cdot;c), \beta_{e,n}(c)\}_{n=0}^\infty \cup \{\psi_{o,n}(\cdot;c), i\beta_{o,n}(c)\}_{n=0}^\infty$ that satisfies
\begin{eqnarray}\label{prop operator eigenfunction gpswf A}
\alpha_{n}(c) \psi_n(p;c) = \int_{A} e^{i c p \cdot p'} \psi_n(p';c) \ind p',\quad p \in A.
\end{eqnarray}
The eigenfunctions belong to $L^2(A)$ and are real-valued, orthogonal, and either even (in which case the eigenvalue is real) or odd (in which case the eigenvalue is purely imaginary). Every eigenvalue $\alpha_{n}(c)$ is non-zero due to an energy argument (using Fourier transform) \cite[pp. 3012]{Slepian64}.  Here in this general symmetric   case we have associated  a single subscript $n$ with the use of $\psi_n$  to avoid any danger of confusion.

By extending  the domain of   $\psi_{n}(\cdot;c)$ to $\mathbb{R}^2$  naturally via
\begin{eqnarray*}
\psi_n(p;c) :=\frac{1}{\alpha_{n}(c)} \int_{A} e^{i c p \cdot p'} \psi_n(p';c) \ind p',\quad p \in \mathbb{R}^2,  \quad \forall n\in \mathbb{N},
\end{eqnarray*}
one can find the double orthogonality \cite[p. 3013]{Slepian64}
 \begin{eqnarray} \label{prop operator eigenfunction gpswf A double orthogonality eqn 1}
 \int_{A}  \psi_m(p;c) \psi_n(p;c) \ind p = \left( \frac{c}{2\pi} \right)^2 |\alpha_{n}(c)|^2  \int_{\mathbb{R}^2}  \psi_m(p;c) \psi_n(p;c) \ind p, \quad \forall m,n\in \mathbb{N},
\end{eqnarray}
and we normalize each eigenfunction such that its energy in $\mathbb{R}^2$ is unit or equivalently
 \begin{eqnarray}\label{prop operator eigenfunction gpswf A double orthogonality eqn 2}
 \int_{A}  \psi_m(p;c) \psi_{m'}(p;c)  \ind p =  \left( \frac{c}{2\pi} \right)^2 |\alpha_m(c)|^2 \delta_{mm'}, \quad \forall m,m'\in \mathbb{N}.
\end{eqnarray}
\section{Picard Criterion and analytic extrapolation}\label{Picard Criterion and analytic extrapolation}
In this section, we study the Picard criterion for reconstructing the contrast from the  full aperture Born data and show that the reconstruction formula can be understood from the viewpoint of data processing and analytic extrapolation.

We begin with a formulation of the inverse problem. Recall in Section \ref{section introduction} that we aim to determine the contrast $q \in L^2(\Omega)$ from 
\begin{equation*}
\{{ u_{b}^{\infty}}(p;k): p \in B(0,2)\}
\end{equation*}
where ${ u_{b}^{\infty}}(p;k)$ is given by \eqref{FM near field case 2layer filtered data def 1 born volume eqv}, i.e,
\begin{eqnarray*}
{ u_{b}^{\infty}}(p;k)  = \int_\Omega e^{i k p\cdot p'} q(p')   \ind p', \quad \forall p \in B(0,2).
\end{eqnarray*}
Now we choose a positive constant $c_F:=c_F(\Omega,k)$ depending on $\Omega$ and $k$ such that  $\Omega \subset D_F$ with $D_F:=\{ \frac{c_F}{2k}x: x\in B(0,1)\}$. Let $  {  u^{\infty}_{b,F}} \in L^2(D_F)$ be given by 
\begin{eqnarray}  \label{section full aperture ubF}
{ u_{b,F}^{\infty}}(p) := \int_{D_F} e^{i \frac{4k^2}{c_F} p \cdot p'}  \underline{q}(p')   \ind p', \quad \forall p \in D_F,
\end{eqnarray}
and let $\underline{q}$ be the extension of $q$ 
\begin{equation}  \label{section full aperture underline q def}
\underline{q}(x):=\left\{
\begin{array}{cc}
q(x)  &  x\in\Omega    \\
0  &      x \not\in \Omega
\end{array}
\right., \quad a.e. \quad x \in \mathbb{R}^2.
\end{equation}
Then the knowledge of $\{{ u_{b}^{\infty}}(p;k): p \in B(0,2)\}$ amounts to the knowledge of $\{u_{b,F}^{\infty}(p): p \in  D_F\}$. Now the inverse problem is formulated as follows.

\textbf{Formulation of the inverse problem}: determine the contrast $\underline{q} \in  L^2(D_F)$ from $\{u_{b,F}^{\infty}(p): p \in  D_F\}$.

We next introduce a suitable eigensystem for the study of  \eqref{section full aperture ubF}. For the chosen positive constant $c_F$, recalling in Section \ref{section GPSWF disk} that $\{\psi_{m,n,\ell }(\cdot;c_F),\alpha_{m,n }(c_F)\}_{m,n \in \mathbb{N}}^{\ell  \in \mathbbm{I}(m)}$ is an eigensystem in $L^2(B(0,1))$, we introduce an orthogonal, complete  set $\{ \psi^F_{m,n,\ell } \}_{m,n \in \mathbb{N}}^{\ell  \in \mathbbm{I}(m)}$ in $L^2(D_F)$ by  
\begin{eqnarray}\label{section full aperture def gpswf psi^F}
 \psi^F_{m,n,\ell}(x;c_F) :=  \frac{2k}{c_F} \psi_{m,n,\ell}\left(\frac{2k}{c_F}x;c_F \right), \quad x \in \mathbb{R}^2.
\end{eqnarray}
It is then verified by a change of variable together with  \eqref{prop operator eigenfunction gpswf B01 extension}  that
\begin{eqnarray}\label{prop operator eigenfunction gpswf psi_c/2k}
 \int_{D_F} e^{i  \frac{4k^2}{c_F}  p \cdot p' } \psi^F_{m,n,\ell}(p' ;c_F) \ind p' =
  \left(\frac{c_F}{2k}\right)^2 \alpha_{m,n}(c_F) \psi^F_{m,n,\ell}(p;c_F), \quad p \in \mathbb{R}^2,
\end{eqnarray}
and by the double orthogonality from \eqref{GPSWF def B01 psi double orthogonal eqn 2} that
\begin{eqnarray}
\int_{\mathbb{R}^2} \psi^F_{m,n,\ell}(p;c_F) \psi^F_{m',n',\ell'}(p;c_F) \ind p &=& \delta_{mm'}\delta_{nn'}\delta_{\ell \ell'},\label{prop double orthogonality psi_c/2k 1}\\
\int_{D_F} \psi^F_{m,n,\ell}(p;c_F) \psi^F_{m',n',\ell'}(p;c_F) \ind p &=&   \left(\frac{c_F}{2\pi} \right)^2  |\alpha_{m,n}(c_F)|^2  \delta_{mm'}\delta_{nn'}\delta_{\ell \ell'}.\label{prop double orthogonality psi_c/2k 2}
\end{eqnarray}

We are now ready to prove the following theorem. Let $\langle \cdot,\cdot \rangle_{D}$ be the $L^2(D)$-inner product for a generic open bounded set $D$ and $\|\cdot\|_{D}$ be the induced norm. When there is no confusion, we drop the subscript for the best presentation.
\begin{theorem} \label{thm chi_q series expansion full aperture}
Let $c_F>0$ be chosen such that   $\Omega \subset D_F=B(0,\frac{c_F}{2k})$.  Let $\underline{q}$ be the extension \eqref{section full aperture underline q def} such that $\underline{q}=q$ in $\Omega$ and $\underline{q}=0$ outside $\Omega$ almost everywhere. Then  $\underline{q}$ is solved by the Picard criterion
\begin{eqnarray} \label{thm chi_q series expansion full aperture representation}
\underline{q}  = \sum_{m,n\in\mathbb{N}}^{\ell  \in \mathbbm{I}(m)}  \left(\frac{2k }{c_F} \right)^2 \frac{1}{  \alpha_{m,n}(c_F)   } \left\langle { u_{b,F}^{\infty}}, \frac{\psi^F_{m,n,\ell}(\cdot;c_F)}{\|\psi^F_{m,n,\ell}(\cdot;c_F)\|} \right\rangle_{D_F}  \frac{\psi^F_{m,n,\ell} (\cdot;c_F)}{\|\psi^F_{m,n,\ell}(\cdot;c_F)\|} 
\end{eqnarray}
where the convergence is in $L^2(D_F)$.
\end{theorem}
\begin{proof}
Note that $\underline{q} \in L^2(D_F)$ and $\{ \psi^{F}_{m,n,\ell}(\cdot;c_F)\}_{m,n \in \mathbb{N}}^{\ell  \in \mathbbm{I}(m)}$ is complete in $L^2(D_F)$; then $\underline{q}$ can be represented by a convergent series in $L^2(D_F)$
\begin{eqnarray*}
\underline{q}(p) = \sum_{m,n\in\mathbb{N}}^{\ell  \in \mathbbm{I}(m)}  q_{m,n,\ell} \frac{\psi^F_{m,n,\ell}(p;c_F)}{\|\psi^F_{m,n,\ell}(\cdot;c_F)\|}, \quad p \in D_F.
\end{eqnarray*}
It suffices to determine $q_{m,n,\ell}$. Note that $\psi^F_{m,n,\ell}$ is an eigenfunction according to \eqref{prop operator eigenfunction gpswf psi_c/2k},
we have  from \eqref{section full aperture ubF} that 
\begin{eqnarray*} 
{ u_{b,F}^{\infty}}(p) = \int_{D_F} e^{i \frac{4k^2}{c_F} p \cdot p'}  \underline{q}(p')   \ind p' = \sum_{m,n\in\mathbb{N}}^{\ell  \in \mathbbm{I}(m)}  \left(\frac{c_F}{2k}\right)^2   \alpha_{m,n}(c_F)    q_{m,n,\ell} \frac{\psi^F_{m,n,\ell}(p;c_F)}{\|\psi^F_{m,n,\ell}(\cdot;c_F)\|},
\end{eqnarray*}
where the series converges in $L^2(D_F)$. From the full knowledge of ${ u_{b,F}^{\infty}}(p) $ in $L^2(D_F)$, we have that 
\begin{eqnarray*}
q_{m,n,\ell} &=&   \left(\frac{2k }{c_F} \right)^2 \frac{1}{  \alpha_{m,n}(c_F)   } \left\langle { u_{b,F}^{\infty}}, \frac{\psi^F_{m,n,\ell}(\cdot;c_F)}{\|\psi^F_{m,n,\ell}(\cdot;c_F)\|} \right\rangle_{D_F}.
\end{eqnarray*}
This completes the proof.
\end{proof}

{ It is useful to make a remark on $\Omega$ and $D_F$. The parameter $c_F$ is the product of the radius of the ball $B(0,2k)$ that supports the restricted Fourier transform and the radius of the ball $D_F$ that supports the contrast (i.e., $c_F=2k \frac{c_F}{2k}$). The performance of the reconstruction formula with noisy data may be observed by a stability estimate involving a delicate interplay of $D_F$, $c_F$, and several other parameters (cf. Theorem \ref{section SLP/regularization theorem qdelta-q estimate}). At the current stage, we have not reached an explicit conclusion on how the choice of $D_F$ explicitly determines the performance of reconstructing the original contrast $q$. 
There is  a similar piece of work on contrast reconstruction using a prolate-Galerkin linear sampling method, we refer the reader to \cite{meng23parameter} for related numerical experiments.
}

\subsection{Data processing and analytic extrapolation} \label{section Picard criterion analytic extrapolation}
 The  prolate spheroidal wave functions in one dimension  extend analytically to $\mathbb{R}$, are  doubly orthogonal, and are complete in the class of band-limited functions \cite{Slepian61}. Similarly the eigenfunctions $\{\psi^F_{m,n,\ell}(\cdot;c_F)\}_{m,n\in\mathbb{N}  }^{\ell  \in \mathbbm{I}(m)}$ also  extend analytically to $\mathbb{R}^2$   \eqref{prop operator eigenfunction gpswf psi_c/2k}, are  doubly orthogonal \eqref{prop double orthogonality psi_c/2k 1}--\eqref{prop double orthogonality psi_c/2k 2}, and are complete in the class of band-limited functions in multiple dimensions (which is indicated by \cite{Slepian64} with techniques in \cite{Slepian61}; we include a proof for completeness). This is the key to extrapolating the data from $D_F$ to the whole space $\mathbb{R}^2$. In this section we show how to interpret the reconstruction formula \eqref{thm chi_q series expansion full aperture representation} of Theorem \ref{thm chi_q series expansion full aperture}   from the viewpoint of data processing and analytic extrapolation.
 
The idea is to extrapolate ${ u_{b,F}^{\infty}}$ to the entire $\mathbb{R}^2$ so that  $\underline{q}$ can then be solved by classical inverse Fourier   transform. Let ${  \underline{u}_{b,F}^{\infty}}$ be an extrapolation of ${ u_{b,F}^{\infty}}$ such that ${  \underline{u}_{b,F}^{\infty}}={ u_{b,F}^{\infty}}$ in $L^2(D_F)$.  The natural analytic extrapolation is to require that  ${  \underline{u}_{b,F}^{\infty}}$ has the following form: 
\begin{eqnarray}  \label{section full aperture ubF extrapolation}
{  \underline{u}_{b,F}^{\infty}}(p) = \int_{D_F} e^{i \frac{4k^2}{c_F} p \cdot p'}  \underline{q}(p')   \ind p', \quad \forall p \in \mathbb{R}^2,
\end{eqnarray}
where we have extrapolated ${ u_{b,F}^{\infty}}$ analytically to $\mathbb{R}^2$ according to \eqref{section full aperture ubF}. 

We introduce here the notation of \textit{band-limited function} in the spirit of   \cite{Slepian64,Slepian61}: a function $u \in L^2(\mathbb{R}^2)$ is called a $D_F$ band-limited function if and only if
\begin{eqnarray}   \label{full aperture section extrapalolation band limited function def}
u(p) = \int_{D_F} e^{i \frac{4k^2}{c_F} p \cdot p'}  f(p')   \ind p', \quad \forall p \in \mathbb{R}^2,
\end{eqnarray}
for some $f \in L^2(D_F)$.
The family of band-limited functions enjoys the following property (which is indicated by \cite{Slepian64} with techniques in \cite{Slepian61}; we include a proof for completeness). Recall that $\{\psi^F_{m,n,\ell}(\cdot;c_F) \in L^2(\mathbb{R}^2)\}_{m,n\in\mathbb{N}  }^{\ell  \in \mathbbm{I}(m)}$ is given by \eqref{section full aperture def gpswf psi^F}.
\begin{prop}
$\{\psi^F_{m,n,\ell}(\cdot;c_F) \in L^2(\mathbb{R}^2)\}_{m,n\in\mathbb{N}}^{\ell  \in \mathbbm{I}(m)}$ is complete in the space of $D_F$ band-limited functions in $L^2(\mathbb{R}^2)$.
\end{prop}
\begin{proof}
Suppose that $u \in L^2(\mathbb{R}^2)$ is   a $D_F$ band-limited function given by \eqref{full aperture section extrapalolation band limited function def} for some $f \in L^2(D_F)$. Note that $\{\psi^F_{m,n,\ell}(\cdot;c_F) \in L^2(D_F)\}_{m,n\in\mathbb{N}}^{\ell  \in \mathbbm{I}(m)}$  is complete in $L^2(D_F)$; then $f$ can be represented by a convergent series
\begin{eqnarray*}
f(p) = \sum_{m,n\in\mathbb{N}}^{\ell  \in \mathbbm{I}(m)}  f_{m,n,\ell} \frac{\psi^F_{m,n,\ell}(p;c_F)}{\|\psi^F_{m,n,\ell}(\cdot;c_F)\|}, \quad p \in D_F.
\end{eqnarray*}
This together with \eqref{full aperture section extrapalolation band limited function def} gives that, for all $p \in \mathbb{R}^2$
\begin{eqnarray*}  
u(p) = \sum_{m,n\in\mathbb{N}}^{\ell  \in \mathbbm{I}(m)}  f_{m,n,\ell}  \int_{D_F} e^{i \frac{4k^2}{c_F} p \cdot p'} \frac{\psi^F_{m,n,\ell}(p')}{\|\psi^F_{m,n,\ell}\|}  \ind p' = \sum_{m,n\in\mathbb{N}}^{\ell  \in \mathbbm{I}(m)}  f_{m,n,\ell}   \left(\frac{c_F}{2k}\right)^2 \alpha_{m,n}(c_F) \frac{\psi^F_{m,n,\ell}(p;c_F)}{\|\psi^F_{m,n,\ell}(\cdot;c_F)\|},
\end{eqnarray*}
where in the last step we applied \eqref{prop operator eigenfunction gpswf psi_c/2k}. This series is convergent in $L^2(\mathbb{R}^2)$ due to \eqref{prop double orthogonality psi_c/2k 1}--\eqref{prop double orthogonality psi_c/2k 2} and the fact that the series of $f$ converges. This proves the proposition.
\end{proof}

Following this notation, the extrapolation ${  \underline{u}_{b,F}^{\infty}}$ is the so-called $D_F$ band-limited function. The following proposition states that such a band-limited function can be fully determined by its value restricted in $D_F$.
\begin{prop}
The band-limited function ${  \underline{u}_{b,F}^{\infty}}$ \eqref{section full aperture ubF extrapolation} can be represented by
\begin{eqnarray} \label{full aperture section extrapalolation ubF underline series representation}
{  \underline{u}_{b,F}^{\infty}} (p) = \sum_{m,n\in\mathbb{N}}^{\ell  \in \mathbbm{I}(m)} \frac{1}{(\frac{c_F}{2\pi})^2  |\alpha_{m,n}(c_F)|^2  } \left\langle { u_{b,F}^{\infty}}, \psi^F_{m,n,\ell}(\cdot;c_F) \right\rangle_{D_F}   \psi^F_{m,n,\ell}(p;c_F), \quad p \in \mathbb{R}^2.
\end{eqnarray}
\end{prop}
\begin{proof}
Let the band-limited function ${  \underline{u}_{b,F}^{\infty}}$  be represented uniquely by an infinite series
\begin{eqnarray}\label{section full aperture ubF extrapolation series expansion}
{  \underline{u}_{b,F}^{\infty}} (p) = \sum_{m,n\in\mathbb{N}}^{\ell  \in \mathbbm{I}(m)}  a_{m,n,\ell} \psi^F_{m,n,\ell}(p;c_F), \quad p \in \mathbb{R}^2.
\end{eqnarray}
Projecting ${  \underline{u}_{b,F}^{\infty}}$ \eqref{section full aperture ubF extrapolation series expansion} onto $\psi^F_{m,n,\ell}$ in the disk $D_F$, one can solve for each coefficient $a_{m,n,\ell}$ by the double orthogonality \eqref{prop double orthogonality psi_c/2k 1}--\eqref{prop double orthogonality psi_c/2k 2}  to arrive at \eqref{full aperture section extrapalolation ubF underline series representation}, where we have kept this particular form as each $\psi^F_{m,n,\ell}$ is normalized to have unit norm in $L^2(\mathbb{R}^2)$.
\end{proof}

Now by applying inverse Fourier   transform to \eqref{section full aperture ubF extrapolation}, one can obtain from   \eqref{full aperture section extrapalolation ubF underline series representation} (where in the  second last step we shall apply \eqref{prop operator eigenfunction gpswf psi_c/2k FIT} in the following remark) that
\begin{eqnarray*}  
 \underline{q}(p)&  =& \frac{1}{(2\pi)^2}\int_{\mathbb{R}^2} e^{-i   p \cdot p'} {  \underline{u}_{b,F}^{\infty}}\left(p'\frac{c_F}{4k^2}\right)    \ind p' \nonumber \\
  &=& \frac{1}{(2\pi)^2}\int_{\mathbb{R}^2} e^{-i   p \cdot p'}      \sum_{m,n\in\mathbb{N}}^{\ell  \in \mathbbm{I}(m)} \frac{1}{(\frac{c_F}{2\pi})^2  |\alpha_{m,n}(c_F)|^2  } \langle { u_{b,F}^{\infty}}, \psi^F_{m,n,\ell}(\cdot;c_F) \rangle_{D_F}   \psi^F_{m,n,\ell}\left(p'\frac{c_F}{4k^2};c_F\right) \ind p' \qquad \nonumber \\
  &=& \frac{1}{(2\pi)^2}      \sum_{m,n\in\mathbb{N}}^{\ell  \in \mathbbm{I}(m)} \frac{1}{(\frac{c_F}{2\pi})^2  |\alpha_{m,n}(c_F)|^2  } \langle { u_{b,F}^{\infty}}, \psi^F_{m,n,\ell}(\cdot;c_F) \rangle_{D_F} \frac{1}{ \frac{1}{(2\pi)^2}  \big(\frac{c_F}{2k}\big)^2 \alpha_{m,n}(c_F)  }  \psi^F_{m,n,\ell}(p;c_F) \nonumber\\
  &=& \sum_{m,n\in\mathbb{N}}^{\ell  \in \mathbbm{I}(m)}  \left(\frac{2k }{c_F} \right)^2 \frac{1}{  \alpha_{m,n}(c_F)   } \left\langle { u_{b,F}^{\infty}}, \frac{\psi^F_{m,n,\ell}(\cdot;c_F)}{\|\psi^F_{m,n,\ell}(\cdot;c_F)\|} \right\rangle_{D_F}  \frac{\psi^F_{m,n,\ell}(p;c_F)}{\|\psi^F_{m,n,\ell}(\cdot;c_F)\|}, \quad p \in D_F. 
\end{eqnarray*}
This shows that we get the same result as in Theorem \ref{thm chi_q series expansion full aperture} by extrapolating the data first followed by  applying the inverse Fourier transform.
 \begin{remark}
 In the second last step, we applied
\begin{eqnarray}\label{prop operator eigenfunction gpswf psi_c/2k FIT}
\psi^F_{m,n,\ell}(p';c_F) = \frac{1}{(2\pi)^2}  \left(\frac{c_F}{2k}\right)^2 \alpha_{m,n}(c_F)  \int_{\mathbb{R}^2} e^{- i  p \cdot p' } \psi^F_{m,n,\ell}\left(p\frac{c_F}{4k^2};c_F\right)  \ind p,   \quad p' \in D_F,
\end{eqnarray}
for any $m,n\in \mathbb{N}, \ell  \in \mathbbm{I}(m)$.
This can be proved as follows: let $\mathbb{I}_{D_F}$ be the characteristic function that $\mathbb{I}_{D_F}=1$ in $D_F$ and $\mathbb{I}_{D_F}=0$ outside $D_F$.
Note that \eqref{prop operator eigenfunction gpswf psi_c/2k} gives
\begin{eqnarray*}
 \int_{\mathbb{R}^2 } e^{i  \frac{4k^2}{c_F}  p \cdot p' } \psi^F_{m,n,\ell}(p';c_F ) \mathbb{I}_{D_F}(p') \ind p' =
  \left(\frac{c_F}{2k}\right)^2 \alpha_{m,n}(c_F) \psi^F_{m,n,\ell}(p;c_F), \quad p \in \mathbb{R}^2.
\end{eqnarray*}
Taking the inverse Fourier transform of the above equation  yields  
$$
\psi^F_{m,n,\ell}(p' ;c_F) \mathbb{I}_{D_F}(p') = \frac{1}{(2\pi)^2}  \left(\frac{c_F}{2k}\right)^2 \alpha_{m,n}(c_F)  \int_{\mathbb{R}^2} e^{- i  p \cdot p' } \psi^F_{m,n,\ell}\left(p\frac{c_F}{4k^2};c_F\right)  \ind p,
$$
and this shows \eqref{prop operator eigenfunction gpswf psi_c/2k FIT}.
\end{remark}
 \begin{remark}
 Note that analytic extrapolation is ill-posed in general (cf. \cite{ChengPengYamamoto2005IP,LuXuXu2012AA}), and \eqref{full aperture section extrapalolation ubF underline series representation} indeed performs an analytic extrapolation. Note that $\psi^F_{m,n,\ell}$ is normalized to have unit norm in $L^2(\mathbb{R}^2)$ \eqref{prop double orthogonality psi_c/2k 1}, the ill-posedness of analytic extrapolation \eqref{full aperture section extrapalolation ubF underline series representation}  is revealed by the fact that the eigenvalue $\alpha_{m,n}(c_F)$ goes to zero as $m,n\to \infty$.
 \end{remark}

\section{Regularity, Approximation theory, Regularization and Stability} \label{section full aperture data regularity regularization stability}
In addition to the salient feature  that the data-driven basis $\{\varphi_{m,n,\ell}(\cdot;c)\}_{m,n \in \mathbb{N}}^{\ell  \in \mathbbm{I}(m)}$  extends analytically to $\mathbb{R}^2$, is  doubly orthogonal, and is complete in the class of band-limited functions,   another feature is that $\{\varphi_{m,n,\ell}(\cdot;c)\}_{m,n \in \mathbb{N}}^{\ell  \in \mathbbm{I}(m)}$ is also a basis for a Sturm-Liouville differential operator. This Sturm-Liouville differential operator brings additional regularity estimates that lead to an explicit stability estimate for a spectral cutoff regularization strategy. In this section, we study the relevant Sturm-Liouville theory in Section \ref{section SLP} and regularity estimates and approximation theory in Section \ref{section approximation in L2}, followed by a spectral cutoff regularization and its stability estimate in Section \ref{section full aperture regularization stabiltiy}.
\subsection{Sturm-Liouville theory} \label{section SLP}
Recall that the eigenfunction $\psi_{m,n,\ell}(\cdot;c)$, $m,n \in \mathbb{N},\ell  \in \mathbbm{I}(m)$, for the disk $B(0,1)$ is given by \eqref{section GPSWF disk eqn psi and varphi}, and  Slepian \cite{Slepian64} showed that  the radial part $\varphi_{m,n}(r;c)$  is also the eigenfunction of   the  following (singular) Sturm–Liouville differential operator \eqref{GPSWF def B01 SLP Dcr eqn 1}--\eqref{GPSWF def B01 SLP Dcr eqn 2}, i.e.,
\begin{eqnarray*} 
&& \mathcal{D}_{c,r} \varphi_{m,n}(\cdot;c) = \chi^o_{m,n}(c), \quad \mbox{ where }  \\
&& \mathcal{D}_{c,r}=-(1-r^2) \frac{\ind^2 }{ \ind r^2} + 2r \frac{\ind   }{ \ind r} - \Big(  \frac{1/4-m^2}{r^2} -c^2  r^2 \Big), 
\end{eqnarray*}
and $ \chi^o_{m,n}(c)$ is the corresponding  eigenvalue.
This is a Sturm-Liouville problem in the radial variable $r$. 

Now we have the following Sturm-Liouville problem in the variable $x=(r \cos \theta, r \sin \theta)^T$ where $T$ denotes the transpose.
Let the Sturm-Liouville differential operator  $ \mathcal{D}_{c,x}$  (see also \cite{ZLWZ20})  be given by
\begin{equation*}
 \mathcal{D}_{c,x}:= -(1-r ^2) \partial^2_r + (3r-\frac{1}{r}) \partial_r -\frac{1}{r^2} \partial^2_\theta  +  c^2r^2.
\end{equation*}
\begin{lemma}
Let the eigenfunction $\psi_{m,n,\ell}(\cdot;c)$ be given by \eqref{section GPSWF disk eqn psi and varphi} for all $m,n \in \mathbb{N},\ell  \in \mathbbm{I}(m)$. Then each 
$\psi_{m,n,\ell}(\cdot;c)$ is also the eigenfunction to the following Sturm-Liouville problem:
\begin{eqnarray*}
 \mathcal{D}_{c,x} \psi_{m,n,\ell}(\cdot;c) = \chi_{m,n} (c) \psi_{m,n,\ell}(\cdot;c)  \mbox{ in } B(0,1),
\end{eqnarray*}
where we define the corresponding eigenvalue by $\chi_{m,n}(c)$.
\end{lemma}
\begin{proof}
Note that $\psi_{m,n,\ell}(\cdot;c)$ and $\varphi_{m,n}(\cdot;c)$ are related by \eqref{section GPSWF disk eqn psi and varphi}  for all $m,n \in \mathbb{N},\ell  \in \mathbbm{I}(m)$; then we can directly calculate
\begin{eqnarray*}
&& \mathcal{D}_{c,x} \psi_{m,n,\ell}(\cdot;c) \\
&=&   [-(1-r ^2) \partial^2_r + (3r-\frac{1}{r}) \partial_r -\frac{1}{r^2} \partial^2_\theta  +  c^2r^2] [r^{-\frac{1}{2}}\varphi_{m,n}(\cdot;c)(r) Y_{m,\ell} (\theta)] \\
&=& -r^{-\frac{1}{2}} Y_{m,\ell} \Big[  (1-r^2) \frac{\ind^2 \varphi_{m,n}(\cdot;c)}{ \ind r^2} - 2r \frac{\ind  \varphi_{m,n}(\cdot;c) }{ \ind r} + \Big(  \frac{1/4-m^2}{r^2} -c^2  r^2 +  \frac{3}{4} \Big) \varphi_{m,n}(\cdot;c) \Big] \\
&=& r^{-\frac{1}{2}} Y_{m,\ell} \left(\mathcal{D}_{c,r} - \frac{3}{4}\right)\varphi_{m,n}(\cdot;c).
\end{eqnarray*}
Since $\varphi_{m,n}(\cdot;c)$ is an eigenfunction of $\mathcal{D}_{c,r}$, then $\psi_{m,n,\ell}(\cdot;c)$ is an eigenfunction of $\mathcal{D}_{c,x}$.
This completes the proof.
\end{proof}
The operator $ \mathcal{D}_{c,x}$ is the natural multi-dimensional generalization of the Sturm-Liouville problem in one dimension considered by Slepian \cite{Slepian64}. 
The operator  $ \mathcal{D}_{c,x}$ is self-adjoint (see also \cite{ZLWZ20}) in the sense that for any $u,v$ in $H^1(B(0,1))$, 
\begin{eqnarray*}
(  \mathcal{D}_{c,x}u,v )_{B(0,1)} &=& \langle \nabla u, \nabla v \rangle_{L^2(B(0,1);w)} +   \langle (x_2 \partial_{x_1} - x_1 \partial_{x_2}) u, (x_2 \partial_{x_1} - x_1 \partial_{x_2}) v \rangle_{L^2(B(0,1))} \\
&&+ c^2 \langle x u,xv\rangle_{L^2(B(0,1))}{ ,}
\end{eqnarray*}
where $(\cdot,\cdot)_{B(0,1)}$ denotes the duality paring, 
and $L^2(B(0,1);w)$ is associated with the $w(x):=1-|x|^2$ weighted norm and inner product
$$
\|u\|_{L^2(B(0,1);w)} := \sqrt{ \langle u,u\rangle_{L^2(B(0,1);w)} }, \quad \langle u,v\rangle_{L^2(B(0,1);w)}:=\int_{B(0,1)} u(x) \overline{v}(x) w(x) \ind x.
$$
The operator  $ \mathcal{D}_{c,x}$ is strictly positive, since for any nonzero $u \in H^1(B(0,1))$,
\begin{eqnarray*}
( \mathcal{D}_{c,x}u,u )_{B(0,1)} = \|\nabla u\|^2_{L^2(B(0,1);w)} +   \| (x_2 \partial_{x_1} - x_1 \partial_{x_2}) u\|^2_{L^2(B(0,1))} + c^2 \|x u\|^2_{L^2(B(0,1))}>0.
\end{eqnarray*}
 
The eigensystem  $\{\psi_{m,n,\ell}(\cdot;c), \chi_{m,n}(c) \}_{m,n \in \mathbb{N}}^{\ell  \in \mathbbm{I}(m)}$ satisfies the following important property (cf. \cite[Theorem 3.2]{ZLWZ20}).
\begin{lemma} \label{Lemma chimn asymptotic}
For any $c>0$, it holds that
$\{\psi_{m,n,\ell}(\cdot;c)\}_{m,n \in \mathbb{N}}^{\ell  \in \mathbbm{I}(m)}$ is a complete orthogonal basis in $L^2(B(0,1))$ and
\begin{eqnarray*}
(m+2n)(m+2n+2) { <} \chi_{m,n}(c) <  (m+2n)(m+2n+2)+c^2.
\end{eqnarray*}
\end{lemma}

Now we  introduce the Sobolev space for any integer $s>0$ (see \cite{wang10} for the one-dimensional case)
\begin{eqnarray*}
\widetilde{H}^s_c(B(0,1)) := \Bigg\{ u \in L^2(B(0,1)):  \sum_{m,n \in \mathbb{N}}^{\ell  \in \mathbbm{I}(m)}  \chi^s_{m,n}(c)  \left| \left\langle u, \frac{\psi_{m,n,\ell}(\cdot;c)}{\|\psi_{m,n,\ell}(\cdot;c)\|} \right\rangle \right|^2 < \infty     \Bigg\},  
\end{eqnarray*}
equipped with the following   energy norm in $\widetilde{H}^s_c(B(0,1))$, $\forall s=1,2,\cdots$ (cf. \cite[Theorem 2.37 and Corrollary 2.38]{mclean00}):
\begin{eqnarray*}
\|u\|^2_{\widetilde{H}^s_c(B(0,1))} :=  \sum_{m,n \in \mathbb{N}}^{\ell  \in \mathbbm{I}(m)}  \chi^s_{m,n}(c)  \left| \left\langle u, \frac{\psi_{m,n,\ell}(\cdot;c)}{\|\psi_{m,n,\ell}(\cdot;c)\|} \right\rangle \right|^2 = \int_{B(0,1)} \big(D^s_{c,x}u\big) \overline{u} \ind x.
\end{eqnarray*}
Now { we define} the Sobolev space $\widetilde{H}^s_c(B(0,1))$ for any real $s>0$ by interpolation as in \cite[Exercise B.8, p.~333 ]{mclean00}, and the norm is characterized by  
\begin{eqnarray*}
\|u\|^2_{\widetilde{H}^s_c(B(0,1))} =  \sum_{m,n \in \mathbb{N}}^{\ell  \in \mathbbm{I}(m)}  \chi^s_{m,n}(c)  \left| \left\langle u, \frac{\psi_{m,n,\ell}(\cdot;c)}{\|\psi_{m,n,\ell}(\cdot;c)\|} \right\rangle \right|^2, \quad \forall s>0.
\end{eqnarray*}
\subsection{Regularity estimate and approximation theory} \label{section approximation in L2}
We now estimate the $L^2(B(0,1))$ approximation error for a spectral cutoff approximation of functions in  $H^s(B(0,1))$, $0<s\le 1$. By Lemma \ref{Lemma chimn asymptotic} the eigenvalue $\chi_{m,n}(c)$ approaches to infinity as $m,n\to \infty$; then for a given small value $\alpha>0$, we can define a spectral cutoff approximation for any $u \in L^2(B(0,1))$:
 \begin{eqnarray} \label{section SLP projection u def}
\pi_{\alpha,c} u := \sum_{m,n \in \mathbb{N},\ell  \in \mathbbm{I}(m)}^{\chi_{m,n}(c)<\alpha^{-1}}  \left\langle u, \frac{\psi_{m,n,\ell}(\cdot;c)}{\|\psi_{m,n,\ell}(\cdot;c)\|} \right\rangle  \frac{\psi_{m,n,\ell}(\cdot;c)}{\|\psi_{m,n,\ell}(\cdot;c)\|} .
\end{eqnarray}
We begin with the result on approximations of any function $u \in \widetilde{H}^s_c(B(0,1))$.
\begin{lemma} \label{section SLP lemma u tildeH^1 projection error}
Suppose that $u \in \widetilde{H}^s_c(B(0,1))$ for $s\ge0$. Then 
\begin{eqnarray*}
\|  \pi_{\alpha,c} u - u \| \le \alpha^{s/2} \|u\|_{\widetilde{H}^s_c(B(0,1))}
\end{eqnarray*}
\end{lemma}
\begin{proof}
This proof follows directly from
\begin{eqnarray*}
\|  \pi_{\alpha,c} u - u \|^2 &=&  \sum_{m,n \in \mathbb{N},\ell  \in \mathbbm{I}(m)}^{\chi_{m,n}(c)\ge\alpha^{-1}}  \left| \left\langle u, \frac{\psi_{m,n,\ell}(\cdot;c)}{\|\psi_{m,n,\ell}(\cdot;c)\|} \right\rangle \right|^2 \\
&=&  \sum_{m,n \in \mathbb{N},\ell  \in \mathbbm{I}(m)}^{\chi_{m,n}(c)\ge\alpha^{-1}} \chi^{-s}_{m,n}(c)  \chi^s_{m,n}(c)  \left| \left\langle u, \frac{\psi_{m,n,\ell}(\cdot;c)}{\|\psi_{m,n,\ell}(\cdot;c)\|} \right\rangle \right|^2 \\
&\le & \alpha^s \sum_{m,n \in \mathbb{N},\ell  \in \mathbbm{I}(m)}^{\chi_{m,n}(c)\ge\alpha^{-1}} \chi^s_{m,n}(c)   \left| \left\langle u, \frac{\psi_{m,n,\ell}(\cdot;c)}{\|\psi_{m,n,\ell}(\cdot;c)\|} \right\rangle \right|^2 \le \alpha^s \|u\|^2_{\widetilde{H}^s_c(B(0,1))}. 
\end{eqnarray*}
This completes the proof.
\end{proof}

The following lemma estimates the $\widetilde{H}^s_c(B(0,1))$ norm by the ${H}^s(B(0,1))$ norm.
\begin{lemma} \label{section SLP lemma u tildeH^1 to H^1}
Suppose that $u \in  {H}^s(B(0,1))$ with $0<s\le1$. Then $u \in  \widetilde{H}^s_c(B(0,1))$ and
\begin{equation} \label{section SLP lemma u tildeH^1 to H^1 eqn}
\|u\|_{\widetilde{H}^s_c(B(0,1))} \le C^{s/2} (1+c^2)^{s/2} \|u\|_{H^s(B(0,1))}  
\end{equation}
for some positive constant $C\ge\sqrt{3}$ independent of $s$ and $c$.
\end{lemma}
\begin{proof}
We first prove the case when $s=1$.
This is a direct consequence of 
 \begin{eqnarray*}
&&\|u\|^2_{{ \widetilde{H}^1_c(B(0,1))}} =  \int_{B(0,1)} \big(D_{c,x}u\big) \overline{u} \ind x  \\
&=& \|\nabla u\|_{L^2(B(0,1);w)} +   \| (x_2 \partial_{x_1} - x_1 \partial_{x_2}) u\|^2_{L^2(B(0,1))} + c^2 \|x u\|^2_{L^2(B(0,1))}\\
&\le& \|\nabla u\|^2_{L^2(B(0,1))} + 2\|\nabla u\|^2_{L^2(B(0,1))} + c^2\| u\|^2_{L^2(B(0,1))} \\
&\le& (3+c^2) \|u\|^2_{H^1(B(0,1))},
\end{eqnarray*}
which yields that
 \begin{eqnarray} \label{section SLP lemma u tildeH^1 to H^1 proof eqn 1}
\|u\|_{{ \widetilde{H}^1_c}(B(0,1))} \le C (1+c^2)^{1/2} \|u\|_{H^1(B(0,1))},
\end{eqnarray}
for some positive constant $C\ge\sqrt{3}$ independent of $s$ and $c$.
Note that $\widetilde{H}^s_c(B(0,1))$ for any real $s>0$ is given by interpolation as in \cite[Exercise B.8, p.~333]{mclean00}, and ${H}^s(B(0,1))$ for any real $s>0$ is given by interpolation as in \cite[Theorem B.8, p.~330]{mclean00}. Now an application of interpolation of Sobolev spaces  \cite[Theorem B.2, p.~320]{mclean00} together with estimate \eqref{section SLP lemma u tildeH^1 to H^1 proof eqn 1} yields \eqref{section SLP lemma u tildeH^1 to H^1 eqn}. This completes the proof.
\end{proof}
\begin{remark}
The regularity estimate in Lemma \ref{section SLP lemma u tildeH^1 to H^1} is of independent interest in approximation theory using the generalized prolate spheroidal wave functions and the data-driven basis.
\end{remark}

Now we are ready to prove the following theorem.
\begin{theorem} \label{section SLP theorem u H^1 projection error}
Suppose that $u \in  {H}^s(B(0,1))$ with $0<s\le1$. Then 
\begin{eqnarray*}
\|  \pi_{\alpha,c} u - u \| \le (\alpha C)^{s/2}    (1+c^2)^{s/2}  \|u\|_{H^s(B(0,1))},
\end{eqnarray*}
for some positive constant $C\ge\sqrt{3}$ independent of $\alpha$, $s$ and $c$.
\end{theorem}
\begin{proof}
This is a direct consequence of Lemma \ref{section SLP lemma u tildeH^1 projection error} and Lemma \ref{section SLP lemma u tildeH^1 to H^1}. This completes the proof.
\end{proof}
Now we give a similar result for approximation of any function in $H^s(D_F)$, $0<s\le 1$. For any function $u^F \in L^2(D_F)$, we let
 \begin{eqnarray}\label{section SLP projection uF def}
\pi^F_{\alpha} u^F := \sum_{m,n \in \mathbb{N},\ell  \in \mathbbm{I}(m)}^{\chi_{m,n}(c_F)<\alpha^{-1}}  \left\langle u^F, \frac{\psi^F_{m,n,\ell}(\cdot;c_F)}{\|\psi^F_{m,n,\ell}(\cdot;c_F)\|}\right\rangle_{D_F} \frac{\psi^F_{m,n,\ell}(\cdot;c_F)}{\|\psi^F_{m,n,\ell}(\cdot;c_F)\|}.
\end{eqnarray}
Now we can prove the following corollary.
\begin{corollary} \label{section SLP theorem u H^1 Bc/2k projection error}
Suppose that $u^F \in  {H}^s(D_F)$ with $0<s\le1$. Then 
\begin{eqnarray*}
\left\|  \pi^F_{\alpha} u^F - u^F \right\|_{L^2\left(D_F\right)} \le (\alpha C)^{s/2}   (1+c_F^2)^{s/2}\left(1+\frac{c_F}{2k}\right)^s  \left\|u^F\right\|_{H^s\left(D_F\right)},
\end{eqnarray*}
for some positive constant $C\ge\sqrt{3}$ independent of $\alpha$, $s$, and $c_F$.
\end{corollary}
\begin{proof}
Let
$$
u(p) := \frac{c_F}{2k}u^F\left(p\frac{c_F}{2k} \right), \quad p\in B(0,1).
$$ 
Then it follows from \eqref{section SLP projection uF def} and  relation \eqref{section full aperture def gpswf psi^F} between $\psi^F_{m,n,\ell}$ and $\psi_{m,n,\ell}(\cdot;c_F)$ that
\begin{eqnarray*}
&&\frac{c_F}{2k}\left(\pi^F_{\alpha} u^F\right) \left(\frac{c_F}{2k}p' \right) \\
&=& \frac{c_F}{2k}  \sum_{m,n \in \mathbb{N},\ell  \in \mathbbm{I}(m)}^{\chi_{m,n}(c_F)<\alpha^{-1}}  \left\langle u^F, \frac{\psi^F_{m,n,\ell}(\cdot;c_F)}{\|\psi^F_{m,n,\ell}(\cdot;c_F)\|}\right\rangle_{D_F} \frac{\psi^F_{m,n,\ell}\left(\frac{c_F}{2k}p';c_F \right)}{\|\psi^F_{m,n,\ell}(\cdot;c_F)\|_{D_F}} \\
&=&  \sum_{m,n \in \mathbb{N},\ell  \in \mathbbm{I}(m)}^{\chi_{m,n}(c_F)<\alpha^{-1}}  \left\langle u, \frac{\psi_{m,n,\ell}(\cdot;c_F)}{\|\psi_{m,n,\ell}(\cdot;c_F)\|}\right\rangle_{B(0,1)} \frac{\psi_{m,n,\ell}\left(p';c_F \right)}{\|\psi_{m,n,\ell}(\cdot;c_F)\|_{B(0,1)}}  = \left(\pi_{\alpha,c_F} u\right)(p').
\end{eqnarray*}
This allows us to derive that
 \begin{eqnarray*}
&&\left\|  \pi^F_{\alpha} u^F - u^F \right\|^2_{L^2\left(D_F\right)}  =  \int_{D_F} \left|u^F(p) - (\pi^F_{\alpha} u^F) (p)\right|^2 \ind p \\
&=& \int_{B(0,1)} \left|u^F\left(\frac{c_F}{2k}p'\right) - (\pi^F_{\alpha} u^F) \left(\frac{c_F}{2k}p'\right) \right|^2 \left(\frac{c_F}{2k}\right)^2 \ind p' =\|  \pi_{\alpha,c_F} u - u \|^2_{L^2(B(0,1))}.
\end{eqnarray*}
Together with Theorem \ref{section SLP theorem u H^1 projection error}, by setting $c=c_F$, we have that
 \begin{eqnarray*}
&&\left\|  \pi^F_{\alpha} u^F - u^F \right\|_{L^2\left(D_F\right)} \le   (\alpha C)^{s/2}  (1+c_F^2)^{s/2}\|u\|_{H^s(B(0,1))} \\
&\le& (\alpha C)^{s/2} (1+c_F^2)^{s/2}\left(1+\frac{c_F}{2k}\right)^s  \|u^F\|_{H^s\left(D_F\right)},
\end{eqnarray*}
where in the last step we have applied 
$$
\|u\|_{H^s(B(0,1))} \le \left(1+\frac{c_F}{2k}\right)^s  \|u^F\|_{H^s\left(D_F\right)},
$$
which is due to the interpolation property of Sobolev spaces, $\|u\|_{H^0(B(0,1))} = \|u^F\|_{H^0\left(D_F\right)} $ and 
$$
\|u\|_{H^1(B(0,1))} = \sqrt{ \|u^F\|^2_{L^2\left(D_F\right)} + \Big(\frac{c_F}{2k}\Big)^2\|\nabla u^F \|^2_{L^2\left(D_F\right)} } \le \left(1+\frac{c_F}{2k}\right) \|u^F\|_{H^1\left(D_F\right)}.
$$
This completes the proof.
\end{proof}
With the above regularity estimates, we are ready to study the following regularization strategy and stability estimate.
\subsection{Regularization  and stability} \label{section full aperture regularization stabiltiy}
In practice, the data is only known up to an error $\delta$ with $\|{    u_{b,F}^{\infty,\delta}} - { u_{b,F}^{\infty}}\|_{L^2(D_F)} \le \delta$. In the following we study a regularization strategy based on spectral cutoff. To begin with, for a small value $\alpha>0$ let the operator $\mathcal{R}_{\alpha}^F:L^2(D_F) \to L^2(D_F)$ be given by
  \begin{eqnarray} \label{section SLP/regularization R alpha def}
\mathcal{R}_{\alpha}^F u^F = \sum_{\mathbb{J}(\alpha)}   \frac{1}{(\frac{c_F}{2k})^2 \alpha_{m,n}(c_F) } \left\langle u^F, \frac{\psi^F_{m,n,\ell}(\cdot;c_F)}{\|\psi^F_{m,n,\ell}(\cdot;c_F)\|}\right\rangle_{D_F} \frac{\psi^F_{m,n,\ell}(\cdot;c_F)}{\|\psi^F_{m,n,\ell}(\cdot;c_F)\|}
\end{eqnarray}
for all $u^F \in L^2(D_F)$, where { $\mathbb{J}(\alpha):=\{m,n \in \mathbb{N},\ell  \in \mathbbm{I}(m): \chi_{m,n}(c_F)<\alpha^{-1} \}$.} We also introduce 
\begin{equation}  \label{section SLP/regularization R beta(alpha) def}
\beta(\alpha):= \min_{(m,n,\ell) \in \mathbb{J}(\alpha)} \left\{\left(\frac{c_F}{2k}\right)^2 \left|\alpha_{m,n}(c_F)\right| \right\}.
\end{equation}
Let  $q^{\delta,\alpha}$ be a regularized solution given by
\begin{eqnarray} \label{section SLP/regularization lemma qdelta}
q^{\delta,\alpha} &:=& \mathcal{R}_{\alpha}^F {    u_{b,F}^{\infty,\delta}} \nonumber \\
&=& \sum_{\mathbb{J}(\alpha)}   \frac{1}{(\frac{c_F}{2k})^2 \alpha_{m,n}(c_F) } \left\langle {    u_{b,F}^{\infty,\delta}}, \frac{\psi^F_{m,n,\ell}(\cdot;c_F)}{\|\psi^F_{m,n,\ell}(\cdot;c_F)\|}\right\rangle_{D_F} \frac{\psi^F_{m,n,\ell}(\cdot;c_F)}{\|\psi^F_{m,n,\ell}(\cdot;c_F)\|}.
\end{eqnarray}
We first prove the following lemma.
\begin{lemma} \label{section SLP/regularization lemma qdelta-q estimate}
Let $c_F$ be chosen such that   $\Omega \subset D_F=B(0,\frac{c_F}{2k})$. Suppose that $\|{    u_{b,F}^{\infty,\delta}} - { u_{b,F}^{\infty}}\|_{L^2(D_F)} \le \delta$. Let  $q^{\delta,\alpha}$ be a regularized solution given by \eqref{section SLP/regularization lemma qdelta} and $\beta(\alpha)$ be given by \eqref{section SLP/regularization R beta(alpha) def}. Then it holds that
\begin{eqnarray}\label{section SLP/regularization lemma qdelta-q estimate eqn}
\|q^{\delta,\alpha} - \underline{q}\|_{L^2\left(D_F\right)} \le  \frac{\delta}{\beta(\alpha)} +  \|\pi^F_\alpha \underline{q} - \underline{q}\|_{L^2\left(D_F\right)}.
\end{eqnarray}
\end{lemma}
\begin{proof}
First we observe that
\begin{eqnarray}
\|q^{\delta,\alpha} - \underline{q}\|_{L^2(D_F)}&=& \|\mathcal{R}_{\alpha}^F  {    u_{b,F}^{\infty,\delta}} -\mathcal{R}_{\alpha}^F  { u_{b,F}^{\infty}} + \mathcal{R}_{\alpha}^F { u_{b,F}^{\infty}} - \underline{q}\|_{L^2(D_F)} \nonumber  \\
&\le& \|\mathcal{R}_{\alpha}^F  {    u_{b,F}^{\infty,\delta}} -\mathcal{R}_{\alpha}^F  { u_{b,F}^{\infty}}\|_{L^2(D_F)} + \|\mathcal{R}_{\alpha}^F { u_{b,F}^{\infty}} - \underline{q}\|_{L^2(D_F)}. \label{section SLP/regularization lemma qdelta-q estimate proof eqn 1}
\end{eqnarray}
Now we estimate each of the terms in the right hand side of the above inequality.

It is directly seen that
\begin{eqnarray}\label{proof full aperture Ru-chiq estimate eqn 2}
\|\mathcal{R}_{\alpha}^F  {    u_{b,F}^{\infty,\delta}} -\mathcal{R}_{\alpha}^F  { u_{b,F}^{\infty}}\|_{L^2(D_F)} \le \|\mathcal{R}_{\alpha}^F \|  \|{    u_{b,F}^{\infty,\delta}} -  { u_{b,F}^{\infty}}\|_{L^2(D_F)}  \le  \delta  \|\mathcal{R}_{\alpha}^F \| \le   \frac{\delta}{\beta(\alpha)}
\end{eqnarray}
where in the last step we have applied $\|\mathcal{R}_{\alpha}^F \|  \le 1/\beta$ due to \eqref{section SLP/regularization R alpha def} and that $ \big(\frac{c_F}{2k}\big)^2 |\alpha_{m,n}(c_F)|\ge \beta(\alpha)$ for $(m,n,\ell) \in \mathbb{J}(\alpha)$; here we note that $\beta(\alpha)$ is given by \eqref{section SLP/regularization R beta(alpha) def}.

For the second estimate, we observe from \eqref{thm chi_q series expansion full aperture representation} that 
\begin{eqnarray*}
 \|\mathcal{R}_{\alpha}^F { u_{b,F}^{\infty}} - \underline{q}\|_{L^2(D_F)} &=& \Big\|\sum_{\mathbb{J}(\alpha)} \langle \underline{q}, \psi^F_{m,n,\ell}/\|\psi^F_{m,n,\ell}\|\rangle_{D_F} \psi^F_{m,n,\ell}/\|\psi^F_{m,n,\ell}\| - \underline{q} \Big\|_{L^2(D_F)} \\
 &\le& \|\pi^F_\alpha \underline{q} - \underline{q}\|_{L^2(D_F)}.
\end{eqnarray*}
This  completes the proof.
\end{proof}
Now we are ready to prove the main theorem.
{ 
\begin{theorem} \label{section SLP/regularization theorem qdelta-q estimate}
Let $c_F$ be chosen such that   $\Omega \subset D_F=B(0,\frac{c_F}{2k})$. Suppose that $\|{    u_{b,F}^{\infty,\delta}} - { u_{b,F}^{\infty}}\|_{L^2(D_F)} \le \delta$. Let  $q^{\delta,\alpha}$ be a regularized solution given by
 \eqref{section SLP/regularization lemma qdelta} and $\beta(\alpha)$ be given by \eqref{section SLP/regularization R beta(alpha) def}. If $\underline{q} \in  {H}^s(D_F)$ with $0<s<1/2$,  then it holds that
\begin{eqnarray}\label{section SLP/regularization theorem qdelta-q estimate eqn}
\|{ q^{\delta,\alpha}} - \underline{q}\|_{L^2\left(D_F\right)} \le  \frac{\delta}{{ \beta(\alpha)}} + (\alpha C) ^{s/2} (1+c_F^2)^{s/2}\left(1+\frac{c_F}{2k}\right)^s  \| \underline{q}\|_{H^s\left(D_F\right)},
\end{eqnarray}
where $C\ge\sqrt{3}$ is a positive constant independent of $\delta$, $\alpha$, $s$, and $c_F$.
\end{theorem}
}
\begin{proof}
From Lemma \ref{section SLP/regularization lemma qdelta-q estimate}, it is sufficient to estimate $\|\pi^F_\alpha \underline{q} - \underline{q}\|_{L^2(D_F)}$. Note that $\underline{q} \in  {H}^s(D_F)$; then by Corollary \ref{section SLP theorem u H^1 Bc/2k projection error}, 
\begin{eqnarray*}
\|  \pi^F_\alpha \underline{q} - \underline{q} \|_{L^2(D_F)} \le (\alpha C)^{s/2}  (1+c_F^2)^{s/2}\left(1+\frac{c_F}{2k}\right)^s  \| \underline{q}\|_{H^s(D_F)}.
\end{eqnarray*}
{ This proves \eqref{section SLP/regularization theorem qdelta-q estimate eqn} and completes the proof.
}
\end{proof}

{  
\begin{remark}
The regularity assumption $\underline{q} \in  {H}^s(D_F)$ with $0<s<1/2$ is due to the fact that $\underline{q}$ is the extension of $q$ by $0$. In the case that $\Omega$ coincides with $D_F$, it is feasible to have  the regularity assumption that $q=\underline{q} \in  {H}^s(D_F)$ with any $0<s\le 1$  where the main result still holds; it is also possible to generalize the result for any $s>0$.
\end{remark}
\begin{remark}
Due to the asymptotic $
(m+2n)(m+2n+2) < \chi_{m,n}(c) <  (m+2n)(m+2n+2)+c^2$ in Lemma \ref{Lemma chimn asymptotic}, one may obtain an explicit estimate on the set $\mathbb{J}(\alpha)=\{m,n \in \mathbb{N},\ell  \in \mathbbm{I}(m): \chi_{m,n}(c_F)<\alpha^{-1}\}$ by 
$$
\alpha < \chi_{m,n}^{-1}(c_F) < \frac{1}{(m+2n)^2}, \quad \forall (m,n,\ell) \in \mathbb{J}(\alpha).
$$ 
The parameter $\beta(\alpha)$ depends on $\alpha$ via \eqref{section SLP/regularization R beta(alpha) def} so that
$$
\beta(\alpha) \le   \left\{\left(\frac{c_F}{2k}\right)^2 \left|\alpha_{m,n}(c_F)\right| \right\}, \quad \forall (m,n,\ell) \in \mathbb{J}(\alpha).
$$
Note that  $\alpha_{m,n}(c_F)$ is the eigenvalue of a Fourier type integral operator, so one expects that $\alpha_{m,n}(c_F)$ decay exponentially to zero. Therefore for a given noise level $\delta$, if we choose $\beta(\alpha)=\delta^{\gamma}$ with $0<\gamma<1$, $(\frac{\delta}{{ \beta(\alpha)}},\alpha^{s/2})$ is expected to appear in H\"{o}lder-logarithmic pair, and hence stability estimate \eqref{section SLP/regularization theorem qdelta-q estimate eqn} of Theorem \ref{section SLP/regularization theorem qdelta-q estimate} is expected in the H\"{o}lder-logarithmic type. 
\end{remark}

A further note on the computation. Noting that $\psi_{m,n,\ell}$ is the eigenfunction of the Sturm-Liouville operator, one can compute such an eigenfunction (as well as the corresponding Sturm-Liouville eigenvalue  $\chi_{m,n}(c_F)$) in a very robust way using a Bouwkamp type algorithm \cite{Bouwkamp50} with Legendre polynomial expansion. After one computes the eigenfunction $\psi_{m,n,\ell}$,  $\alpha_{m,n}(c_F)$ can be computed in a very stable manner using its relation to the generalized prolate spheroidal wave functions. We refer the reader to \cite{wang10,ZLWZ20} on these computational aspects.  We further refer the reader to \cite{meng23parameter} for related numerical experiments on a linear sampling method using prolate spheroidal wave functions to achieve parameter identification.
}

\section{Limited-aperture data and multi-frequency data} \label{section limited-aperture multi-frequency data}
The application of the generalized prolate spheroidal wave functions and their related functions is far less developed for the limited-aperture and multi-frequency inverse scattering problems. Remarkably, there also exists a  data-driven basis for a general symmetric set \cite{Slepian64}. In this section we first show that the  Picard criterion also applies to the limited-aperture and multi-frequency cases, and then obtain a  stability estimate for a spectral cutoff regularization.
\subsection{Scaled eigensystem for symmetric set}  \label{section GPSWF symmetry domain}
We first introduce a scaled eigensystem for a generic symmetric bounded open set  $A_h := \{hx: x\in A\}$. Noting that $\{\psi_{n}(\cdot;c), \alpha_{n}(c)\}_{n=0}^\infty$ is an eigensystem given in Section \ref{section GPSWF symmetric}, we introduce another orthogonal complete set  $\{  \widetilde{\psi}_{n}(\cdot;c)\}_{n \in \mathbb{N}}$ in  $L^2(A_h)$ by
\begin{eqnarray*}
 \widetilde{\psi}_{n}(x;c) := h^{-1} \psi_{n}\left(h^{-1}x ;c\right), \quad x\in\mathbb{R}^2.
\end{eqnarray*}
Then by  a change of variable, it follows from \eqref{prop operator eigenfunction gpswf A} that
\begin{eqnarray*}
\int_{A_h} e^{i  \frac{c}{h^2}  p \cdot p' }  \widetilde{\psi}_{n}(p' ;c) \ind p'
&=&  h^2  \alpha_{n}(c)  \widetilde{\psi}_{n}(p;c), \quad p \in A_h,
\end{eqnarray*}
and from \eqref{prop operator eigenfunction gpswf A double orthogonality eqn 1}--\eqref{prop operator eigenfunction gpswf A double orthogonality eqn 2} that 
\begin{eqnarray*} 
\int_{\mathbb{R}^2}  \widetilde{\psi}_{n}(p;c)  \widetilde{\psi}_{n'}(p;c) \ind p &=&\delta_{nn'},\\
\int_{A_h}  \widetilde{\psi}_{n}(p;c)  \widetilde{\psi}_{n'}(p;c) \ind p &=&  \left(\frac{c}{2\pi}\right)^2  |\alpha_{n}(c)|^2  \delta_{nn'}.
\end{eqnarray*}

\subsection{Reconstruction formula}
In Section \ref{subsection reconstruction formula limited multi-frequency formulation IP limited data} and Section \ref{subsection reconstruction formula limited multi-frequency formulation IP multi-frequency data}, we formulate the inverse problems for limited-aperture and multi-frequency data, respectively. In Section \ref{subsection reconstruction formula limited multi-frequency formulation IP Picard criterion} we use the Picard criterion to obtain the reconstruction formula.  
\subsubsection{Limited-aperture data} \label{subsection reconstruction formula limited multi-frequency formulation IP limited data}
We give a formulation of the Born inverse scattering with limited-aperture data. Recall  in  Section \ref{section introduction} that we aim  to determine the contrast $q\in L^2(\Omega)$ from the limited-aperture Born   data
\begin{equation} \label{born medium limited-aperture data continuous}
\{ \widetilde{u}_b^{\infty}(\hat{x};\hat{\theta};k): \hat{x}, \hat{\theta} \in \mathbb{S}_L\}.
\end{equation}
From the representation \eqref{born uinfty integral representation} of $\widetilde{u}_b^{\infty}(\hat{x};\hat{\theta};k)$,    the knowledge of  $\{\widetilde{u}_{b}^{\infty}(\hat{x};\hat{\theta};k): \hat{x}, \hat{\theta} \in \mathbb{S}_L\}$  is equivalent to the knowledge of $\{{ u_{b}^{\infty}}(p;k): p \in L\}$, 
\begin{eqnarray} \label{limited-aperture ub}
{ u_{b}^{\infty}}(p;k)  = \int_\Omega e^{i k p\cdot p'} q(p')   \ind p', \quad \forall p \in L,
\end{eqnarray}
where $L$ is the interior of $\{\hat{\theta}-\hat{x}: \hat{x}, \hat{\theta} \in \mathbb{S}_L\}$ which is symmetric with respect to the origin. 
\begin{remark}
The set $L$ represents the size of the limited-aperture in a more explicit way. Note that as the aperture gets smaller, the area of the set $L$ becomes smaller, which indicates that the inverse problem gets more ill-posed (from the point of view of analytic unique continuation); see Figure \ref{figure plot L} for an illustration. 
\end{remark}
 Suppose that there exists a positive constant $c_L$ such that   $\Omega \subset D_L:= \{c_Lx/k: x\in L\}$. This is guaranteed for aperture size large than one half since the set $L$ contains a small neighborhood of the origin. For aperture size smaller than or equal to one half where the set $L$ does not contain a small neighborhood of the origin (see Figure \ref{figure plot L}), we have to impose  a priori information that there exists $c_L$ such that   $\Omega \subset D_L$ { (this seems to be a restriction on $\Omega$, and it would be interesting to relax such a restriction; for instance, one possible way is to extrapolate the partial data in $L$ to obtain the full aperture data in $B(0,2)$)}. By equation \eqref{limited-aperture ub}, we then have the knowledge of $\{{   u_{b,L}^{\infty}}(p): p \in  D_L$\}  where
\begin{eqnarray} \label{limited-aperture ubL} 
{   u_{b,L}^{\infty}}(p)  := \int_{D_L} e^{i \frac{k^2}{c_L} p \cdot p'}  \underline{q}(p')   \ind p' := (\mathcal{K}^L \underline{q})(p), \quad \forall p \in D_L,
\end{eqnarray}
and we denote the associated operator by $\mathcal{K}^L: L^2(D_L) \to L^2(D_L)$. Now the limited-aperture inverse scattering is formulated as follows.

\textbf{Formulation of the inverse problem}: Assuming that there exists ${ c_L}$ such that $\Omega \subset D_L$, determine the contrast $\underline{q} \in L^2(D_L)$ from $\{{   u_{b,L}^{\infty}}(p): p \in  D_L$\}.

     \begin{figure}[ht!]
\includegraphics[width=0.23\linewidth]{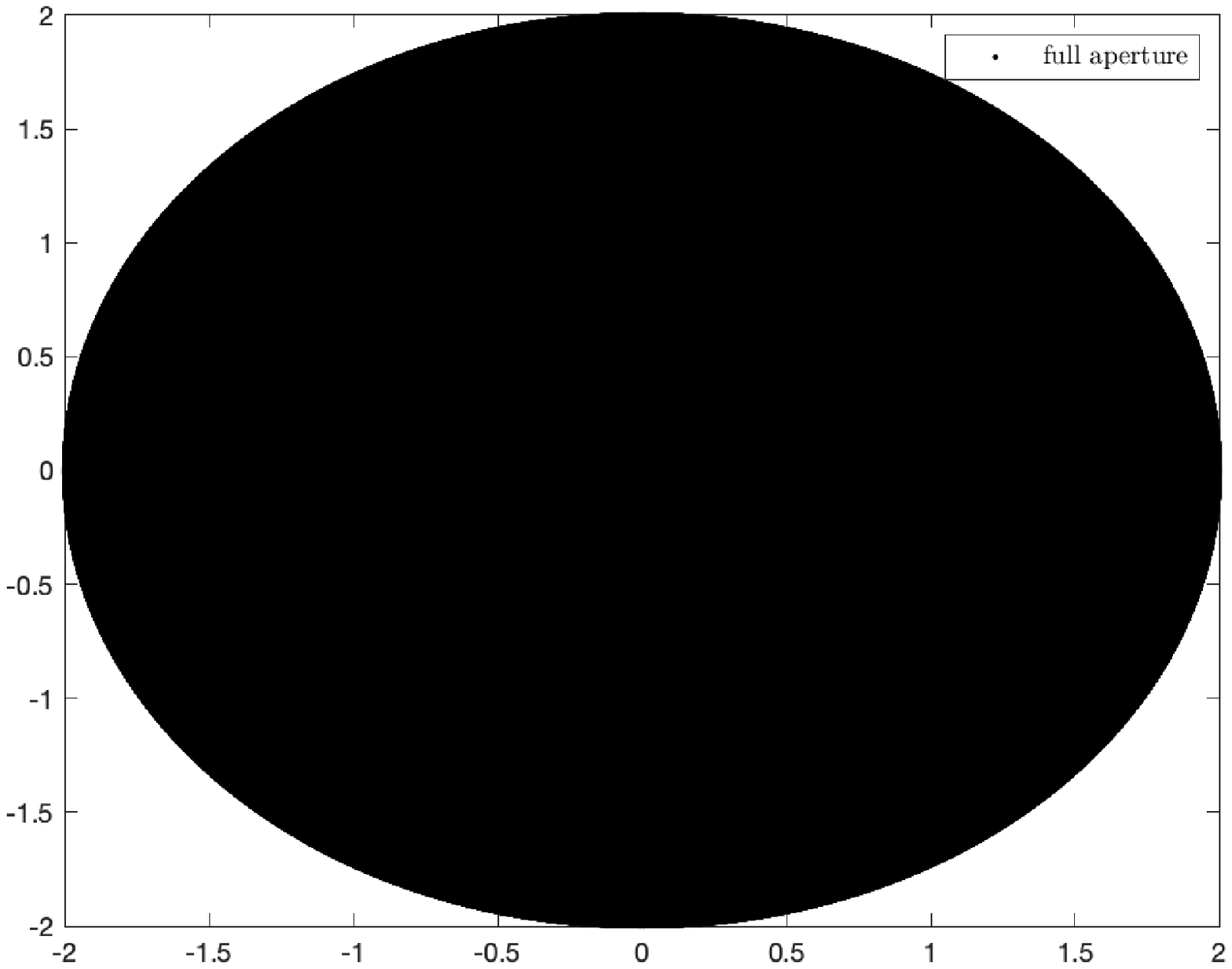}
\includegraphics[width=0.23\linewidth]{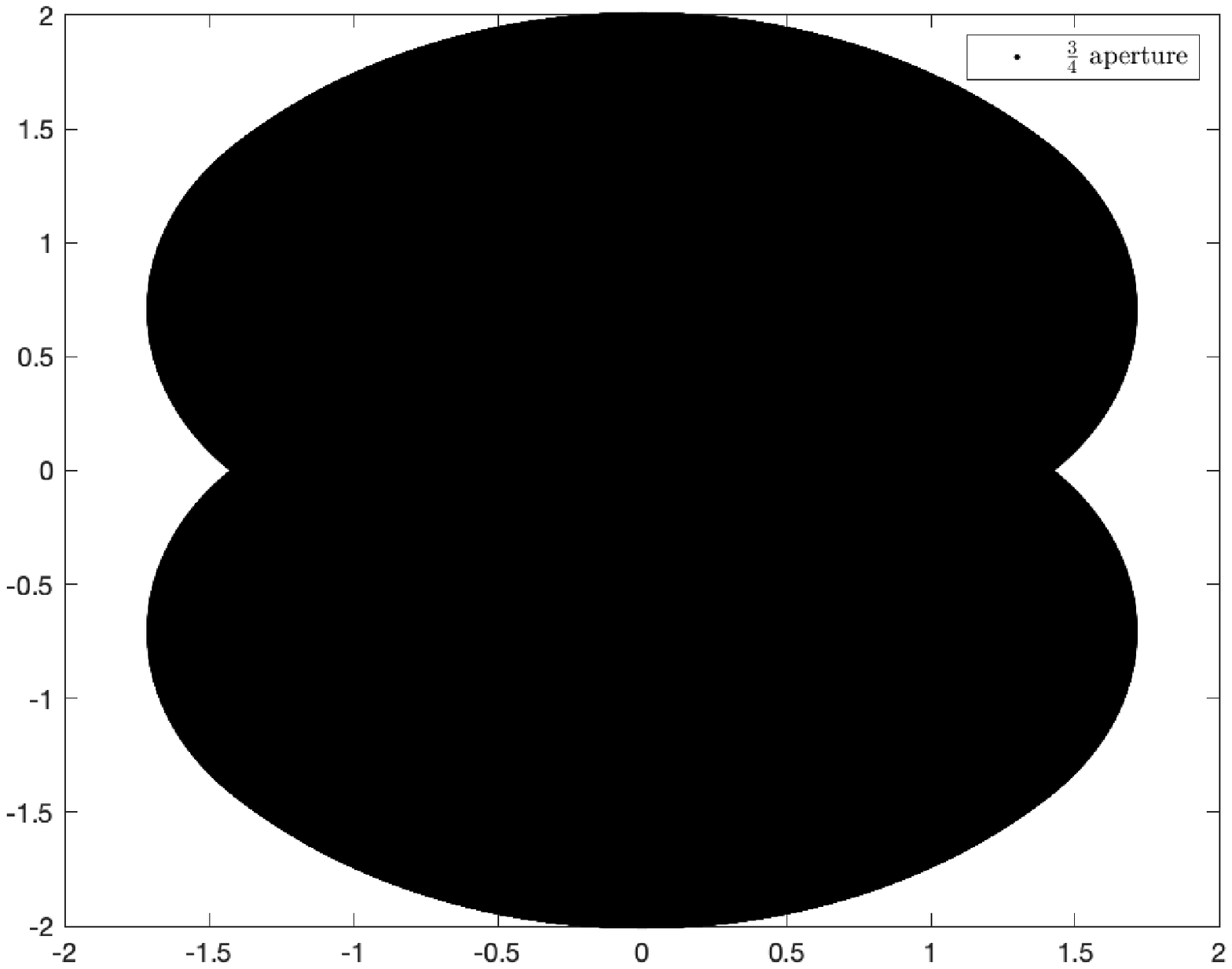}
\includegraphics[width=0.23\linewidth]{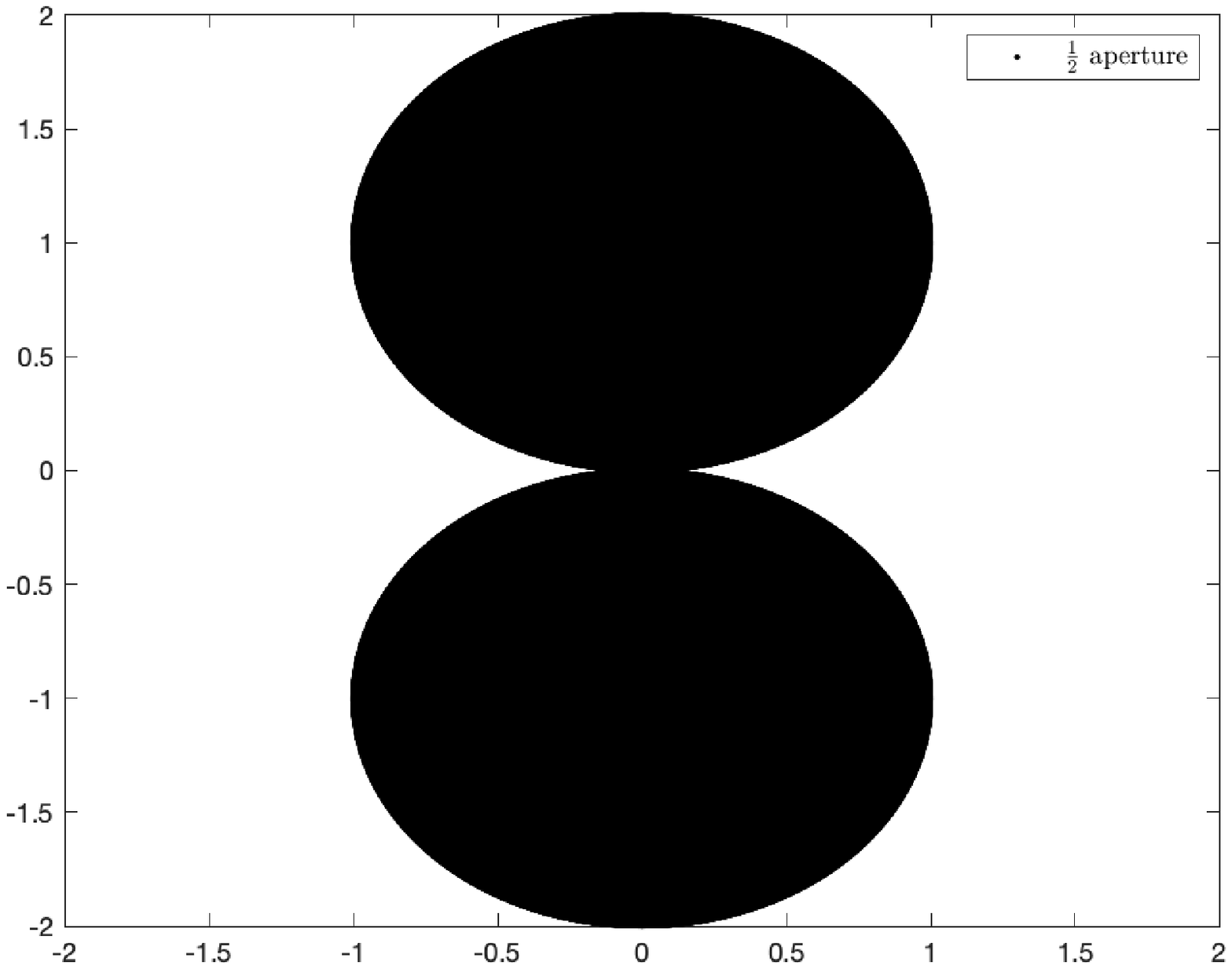}
\includegraphics[width=0.23\linewidth]{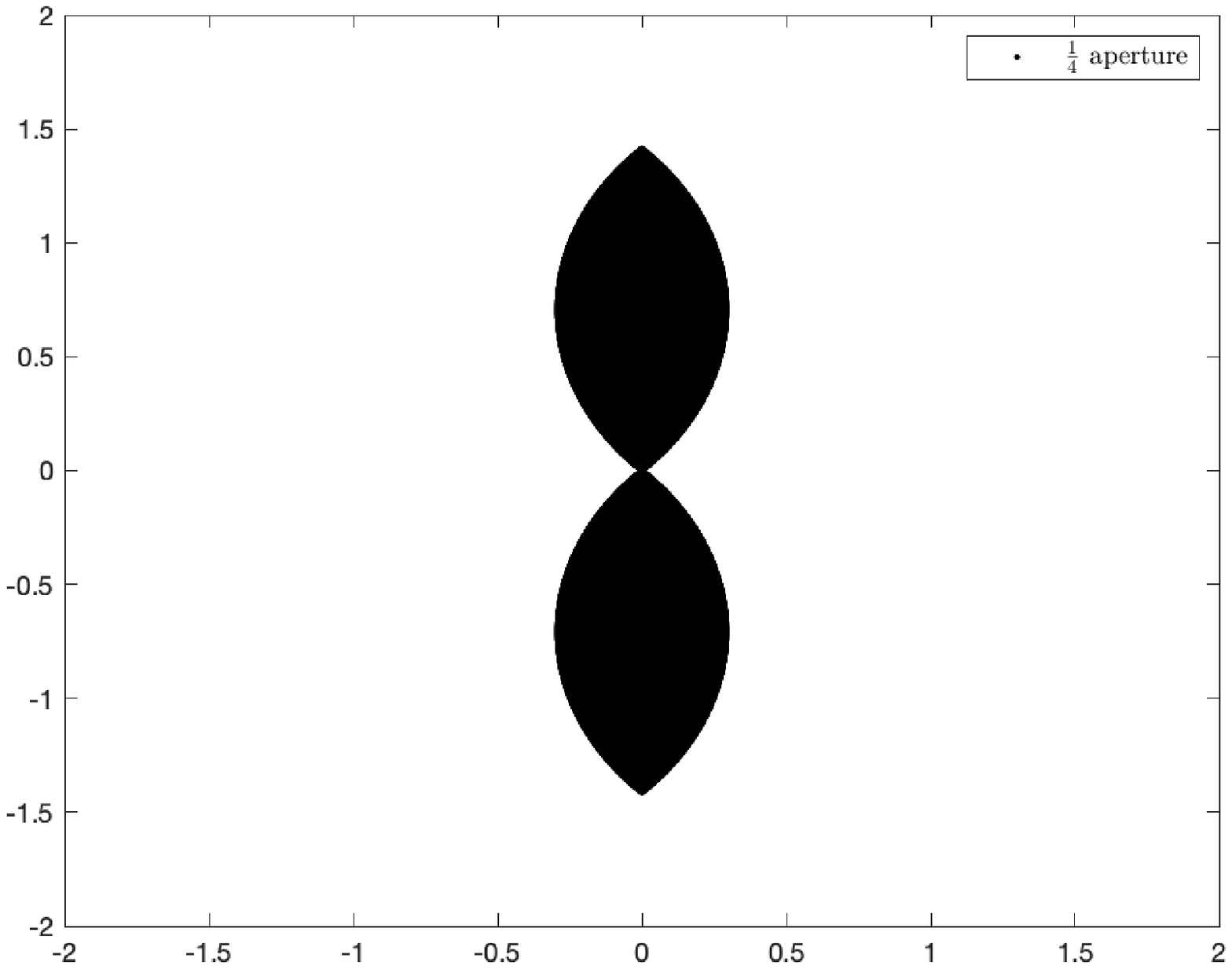}
     \caption{
     \linespread{1}
 The domain $L$ with respect to different aperture sizes. From left to right: full, $\frac{3}{4}$, $\frac{1}{2}$, $\frac{1}{4}$.
     } \label{figure plot L}
    \end{figure}
    
Now we introduce a suitable eigensystem for the study of \eqref{limited-aperture ubL}.  Define the eigensystem $\{ \psi^{L}_{n}(\cdot;c_L), \alpha^L_{n}(c_L)\}_{n \in \mathbb{N}}$ in $L^2(D_L)$ (by setting $A=L$, $c=c_L$, and $h=c_L/k$ in Section \ref{section GPSWF symmetry domain}) such that
\begin{eqnarray*}
\int_{D_L} e^{i  \frac{k^2}{c_L}  p \cdot p' } \psi^L_{n}(p';c_L) \ind p' 
&=&  \big(\frac{c_L}{k}\big)^2  \alpha^L_{n}(c_L) \psi^L_{n}(p;c_L), \quad p \in D_L.
\end{eqnarray*}
In this case the double orthogonality reads
\begin{eqnarray*}
\int_{\mathbb{R}^2} \psi^L_{n}(p;c_L) \psi^L_{n'}(p;c_L) \ind p &=& \delta_{nn'},\\
\int_{D_L} \psi^L_{n}(p;c_L) \psi^L_{n'}(p;c_L) \ind p &=&  \left(\frac{c_L}{2\pi}\right)^2  |\alpha^L_{n}(c_L)|^2 \delta_{nn'}.
\end{eqnarray*}

\subsubsection{Multi-frequency data} \label{subsection reconstruction formula limited multi-frequency formulation IP multi-frequency data}
In the multi-frequency case, recall  in  Section \ref{section introduction} that we aim  to determine  the contrast $q\in L^2(\Omega)$ from the following  Born  data with two opposite observation directions (which are $\pm \hat{x}^*  \in \mathbb{S}$):
 \begin{equation} \label{born medium multi-frequency data two opposite direction continuous}
\{ \widetilde{u}_b^{\infty}(\pm \hat{x}^*; \hat{\theta};k): \hat{\theta} \in \mathbb{S}, k \in (0,K),  K>0\}.
\end{equation}
For this multi-frequency problem, we have that  the knowledge of  $\{\widetilde{u}_{b}^{\infty}(\pm \hat{x}^*;\hat{\theta};k): \hat{\theta}  \in \mathbb{S}, k \in (0,K)\}$  is equivalent to the knowledge of $\{{ u_{b}^{\infty}}(p;K): p \in M \}$ where $M$  is the interior of   $\{a\hat{\theta} \pm a\hat{x}^*: \hat{\theta} \in \mathbb{S}, a \in (0,1)\}$ and
\begin{eqnarray} \label{multi-frequency data ub}
{ u_{b}^{\infty}}(p;K)  = \int_\Omega e^{i K p\cdot p'} q(p')   \ind p', \quad \forall p \in M.
\end{eqnarray}
 Note that $M$ is symmetric with respect to the origin.  Suppose that there exists $c_M$ such that   $\Omega \subset D_M=\{c_My/K: y \in M\}$.  By  \eqref{multi-frequency data ub}, we then have the knowledge of ${   u_{b,M}^{\infty}}(p)$ for $p \in  D_M$  where
\begin{eqnarray} \label{multi-frequency data ubM}
{   u_{b,M}^{\infty}}(p)  := \int_{D_M} e^{i \frac{K^2}{c_M} p \cdot p'}  \underline{q}(p')   \ind p':= (\mathcal{K}^M \underline{q})(p), \quad \forall p \in D_M,
\end{eqnarray}
and we denote the associated operator by $\mathcal{K}^M: L^2(D_M) \to L^2(D_M)$. Now the multi-frequency inverse scattering is formulated as follows.

\textbf{Formulation of the inverse problem}: Assuming that there exists ${ c_M}$ such that $\Omega \subset D_M$, determine the contrast $\underline{q} \in L^2(D_M)$ from $\{{   u_{b,M}^{\infty}}(p): p \in  D_M$\}.

Now we introduce a suitable eigensystem for the study of \eqref{multi-frequency data ubM}. Define the eigensystem $\{ \psi^{M}_{n}(\cdot;c_M), \alpha^M_{n}(c_M)\}_{n \in \mathbb{N}}$ in $D_M$ (by setting $c=c_M$, $h=\frac{c_M}{K}$, and $A=M$ in Section \ref{section GPSWF symmetry domain})  such that
\begin{eqnarray*}
\int_{D_M} e^{i  \frac{K^2}{c_M}  p \cdot p' } \psi^M_{n}(p;c_M) \ind p' 
&=&  \left(\frac{c_M}{K}\right)^2  \alpha^M_{n}(c_M) \psi^M_{n}(p;c_M), \quad p \in D_M.
\end{eqnarray*}
In this case the double orthogonality reads
\begin{eqnarray*}
\int_{\mathbb{R}^2} \psi^M_{n}(p;c_M) \psi^M_{n'}(p;c_M) \ind p &=&  \delta_{nn'},\\
\int_{D_M} \psi^M_{n}(p;c_M) \psi^M_{n'}(p;c_M) \ind p &=&  \left(\frac{c_M}{2\pi}\right)^2  |\alpha^M_{n}(c_M)|^2   \delta_{nn'}.
\end{eqnarray*}
\subsubsection{Picard criterion} \label{subsection reconstruction formula limited multi-frequency formulation IP Picard criterion}
Now we are ready to obtain the following reconstruction formulas by Picard criterion.
\begin{theorem}
 Let $\underline{q}$ be the extension function that $\underline{q}=q$ in $\Omega$ and $\underline{q}=0$ outside $\Omega$ almost everywhere.  
\begin{itemize}
\item 
Suppose that  there exists $c_L$ such that   $\Omega \subset D_L=\{c_Lx/k: x\in L\}$, where $L$ is the interior of $\{\hat{\theta}-\hat{x}: \hat{x}, \hat{\theta} \in \mathbb{S}_L\}$ and $\mathbb{S}_L=\{x: x \in \mathbb{S}, \arg x \in [-\Theta,\Theta],  0<\Theta<\pi\}$. Then $\underline{q}$ is uniquely solved by
\begin{eqnarray*}  
\underline{q}  = \sum_{n\in\mathbb{N}}\left(\frac{k }{c_L} \right)^2 \frac{1}{  \alpha^L_n(c_L)   } \left\langle {   u_{b,L}^{\infty}}, \frac{\psi^L_n(\cdot;c_L)}{\|\psi^L_n(\cdot;c_L)\|} \right\rangle_{D_L}   \frac{\psi^L_n(\cdot;c_L)}{\|\psi^L_n(\cdot;c_L)\|},
\end{eqnarray*}
where the convergence is in $L^2(D_L)$.
\item
Suppose that  there exists $c_M$  such that   $\Omega \subset D_M=\{c_Mx/K: x\in M\}$ where $M$  is the interior of   $\{a\hat{\theta} \pm a\hat{x}^*: \hat{\theta} \in \mathbb{S}, a \in (0,1)\}$. Then $\underline{q}$ is uniquely solved by
\begin{eqnarray*}  
\underline{q} = \sum_{n\in\mathbb{N}}\left(\frac{K}{c_M} \right)^2 \frac{1}{  \alpha_n^M(c_M)   } \left\langle {   u_{b,M}^{\infty}}, \frac{\psi^M_n(\cdot;c_M)}{\|\psi^M_n(\cdot;c_M)\|} \right\rangle_{D_M}   \frac{\psi^M_n(\cdot;c_M)}{\|\psi^M_n(\cdot;c_M)\|},
\end{eqnarray*}
where the convergence is in $L^2(D_M)$.
\end{itemize}
\end{theorem}
\begin{proof}
The proof is exactly the same as the proof of Theorem \ref{thm chi_q series expansion full aperture}.
\end{proof}
\begin{remark}
Noting that the eigenfunctions $\{\psi^L_{n}(\cdot;c_L)\}_{n\in\mathbb{N} }$ and $\{\psi^M_{n}(\cdot;c_M)\}_{n\in\mathbb{N} }$ also extend analytically to $\mathbb{R}^2$, are  doubly orthogonal, and are complete in a proper class of band-limited functions in multiple dimensions, we can  also interpret the above reconstruction formulas   from the viewpoint of data processing and analytic extrapolation (by following Section \ref{section Picard criterion analytic extrapolation} in exactly the same way).
\end{remark}
\subsection{Regularization and stability}
In both the limited-aperture  and the multi-frequency cases, we have considered an operator equation
 \begin{eqnarray} \label{section regularization def K}
\mathcal{K} \tilde{q} = \tilde{u}, \quad \mathcal{K}: L^2(D) \to L^2(D)
\end{eqnarray}
where
\begin{itemize}
\item for the  limited-aperture case: $\mathcal{K} $ is given by \eqref{limited-aperture ubL},   $D=D_L$, and $ \tilde{u} = {   u_{b,L}^{\infty}}$;
\item for the multi-frequency case: $\mathcal{K} $ is given by \eqref{multi-frequency data ubM},   $D=D_M$, and $ \tilde{u} = {   u_{b,M}^{\infty}}$.
\end{itemize}
Therefore we first study the regularization and stability of this abstract equation \eqref{section regularization def K}, and then state the result for each of the   above cases. In each case, we have found an orthogonal eigensystem $\{\phi_n,\mu_n\}_{n\in\mathbb{N}}$ in $L^2(D)$ with real-valued $\phi_n$ such that
\begin{equation} \label{section regularization def eigensystem}
\mathcal{K}  \phi_n = \mu_n \phi_n, \quad \|\phi_n\| = \lambda_n>0,
\end{equation}
$|\mu_n|>0$,  $\lim_{n\to\infty}|\mu_n|=0$,
and $\tilde{q}$ is solved in the form of 
\begin{equation} \label{section regularization representation chi_q}
\tilde{q} = \sum_{n \in \mathbb{N}}  \frac{1}{\mu_n} \left\langle  \tilde{u}, \frac{\phi_n}{\lambda_n} \right\rangle_{L^2(D)} \frac{\phi_n}{\lambda_n}.
\end{equation}

In practice, the data $ \tilde{u}^\delta$ is only known up to an error $\delta$ with $\| \tilde{u}- \tilde{u}^\delta\|_{L^2(D)}\le \delta$ and this motivates us to consider a regularization strategy  based on spectral cutoff where we define $\mathcal{R}_{\alpha}:L^2(D) \to L^2(D)$ by
\begin{equation} \label{section regularization representation chi_q spectral cutoff}
\mathcal{R}_{\alpha}  \tilde{u} = \sum_{|\mu_n|>\alpha} \frac{1}{\mu_n} \left\langle  \tilde{u}, \frac{\phi_n}{\lambda_n} \right\rangle_{L^2(D)} \frac{\phi_n}{\lambda_n}.
\end{equation}

\begin{lemma} \label{theorem full aperture Ru-chiq estimate}
Suppose that $\| \tilde{u}^\delta -  \tilde{u}\|_{L^2(D)} \le \delta$. Let $\tilde{q} \in Range((\mathcal{K} ^*\mathcal{K} )^{\sigma/2})$ for some $\sigma>0$ and $\|(\mathcal{K} ^*\mathcal{K} )^{-\sigma/2} \tilde{q}\|_{L^2(D)} \le E$ for some constant $E$. Let $\alpha(\delta)=c_0 (\delta/E)^{1/(1+\sigma)}$ with some positive constant $c_0$; then it holds that
\begin{eqnarray}\label{theorem full aperture Ru-chiq estimate eqn}
\|\mathcal{R}_{\alpha(\delta)}   \tilde{u}^\delta - \tilde{q}\|_{L^2(D)} \le  \delta^{\sigma/(1+\sigma)} E^{1/(1+\sigma)}\big(1/c_0 + c_0^\sigma  \big).
\end{eqnarray}
\end{lemma}
\begin{proof}
Note that
\begin{eqnarray}
\|\mathcal{R}_{\alpha(\delta)}   \tilde{u}^\delta - \tilde{q}\|_{L^2(D)} &=& \|\mathcal{R}_{\alpha(\delta)}   \tilde{u}^\delta -\mathcal{R}_{\alpha(\delta)}   \tilde{u} + \mathcal{R}_{\alpha(\delta)}   \tilde{u} - \tilde{q}\|_{L^2(D)} \nonumber  \\
&\le& \|\mathcal{R}_{\alpha(\delta)}   \tilde{u}^\delta -\mathcal{R}_{\alpha(\delta)}   \tilde{u}\|_{L^2(D)} + \|\mathcal{R}_{\alpha(\delta)}   \tilde{u} - \tilde{q}\|_{L^2(D)}.\label{proof full aperture Ru-chiq estimate eqn 1}
\end{eqnarray}
Now we estimate each of the terms in the right hand side of the above inequality.

For the first term, it is directly seen that
\begin{eqnarray}\label{proof full aperture Ru-chiq estimate eqn 2}
 \|\mathcal{R}_{\alpha}   \tilde{u}^\delta -\mathcal{R}_{\alpha}   \tilde{u}\|_{L^2(D)} \le \|\mathcal{R}_{\alpha}\| \| \tilde{u}^\delta -  \tilde{u}\|_{L^2(D)} \le  \delta \|\mathcal{R}_{\alpha}\| \le   \frac{\delta}{\alpha}
\end{eqnarray}
where in the last step we have applied $\|\mathcal{R}_{\alpha}\| \le \frac{1}{\alpha}$ due to \eqref{section regularization representation chi_q spectral cutoff}.

For the second term, we observe that  $\|(\mathcal{K}^*\mathcal{K} )^{-\sigma/2} \tilde{q}\|_{L^2(D)} \le E$ implies that
\begin{eqnarray*}
\|(\mathcal{K}^*\mathcal{K} )^{-\sigma/2} \tilde{q}\|_{L^2(D)}^2=\sum_{n \in \mathbb{N}}  |\mu_n|^{-2\sigma} |\langle\tilde{q}, \phi_n/\lambda_n \rangle_{L^2(D)} |^2 \le E^2
\end{eqnarray*}
and thereby
\begin{eqnarray} \label{proof full aperture Ru-chiq estimate eqn 3}
&&\|\mathcal{R}_{\alpha }   \tilde{u} - \tilde{q}\|^2_{L^2(D)} = \Big\| \sum_{|\mu_n|\le \alpha} \frac{\langle  \tilde{u}, \phi_n \rangle }{\lambda^2_n \mu_n} \phi_n \Big\|^2_{L^2(D)} =\Big\| \sum_{|\mu_n|\le \alpha} \langle \tilde{q}, \phi_n/\lambda_n \rangle \frac{\phi_n}{\lambda_n} \Big\|^2_{L^2(D)} \nonumber \\
&=& \sum_{|\mu_n|\le \alpha}  |\mu_n|^{2\sigma} |\mu_n|^{-2\sigma} |\langle \tilde{q}, \phi_n/\lambda_n \rangle |^2 \le \alpha^{2\sigma} \sum_{|\mu_n|\le \alpha}   |\mu_n|^{-2\sigma} |\langle \tilde{q}, \phi_n/\lambda_n \rangle |^2 \nonumber \\
 &\le&  \alpha^{2\sigma} \|(\mathcal{K}^*\mathcal{K} )^{-\sigma/2} \tilde{q}\|_{L^2(D)}^2.
\end{eqnarray}

Finally \eqref{proof full aperture Ru-chiq estimate eqn 1}--\eqref{proof full aperture Ru-chiq estimate eqn 3} yields that
\begin{eqnarray*}
\|\mathcal{R}_{\alpha(\delta)}   \tilde{u}^\delta - \tilde{q}\|_{L^2(D)} &\le&  \frac{\delta}{\alpha(\delta)} +  [\alpha(\delta)]^{\sigma} \|(\mathcal{K}^*\mathcal{K})^{-\sigma/2} \tilde{q}\|_{L^2(D)} \le \frac{\delta}{\alpha(\delta)} +  [\alpha(\delta)]^{\sigma} E.
\end{eqnarray*}
Let $\alpha(\delta) = c_0 (\delta/E)^{1/(1+\sigma)}$ with some positive constant $c_0$; then
$$
\|\mathcal{R}_{\alpha(\delta)}  u^\delta - \tilde{q}\|_{L^2(D)} \le  \frac{\delta}{c_0 (\delta/E)^{1/(1+\sigma)}} +   (c_0 (\delta/E)^{1/(1+\sigma)})^{\sigma} E = \delta^{\sigma/(1+\sigma)} E^{1/(1+\sigma)}\big(1/c_0 + c_0^\sigma  \big).
$$
This proves \eqref{theorem full aperture Ru-chiq estimate eqn} and  completes the proof.
\end{proof}

Note that the regularity assumption on $\tilde{q}$ (i.e., $\tilde{q} \in Range((\mathcal{K} ^*\mathcal{K} )^{\sigma/2})$ for some $\sigma>0$) is less explicit compared to the one (i.e., $\underline{q} \in  {H}^s(D_F)$ with $0<s<1/2$) in the full aperture case in Theorem \ref{section SLP/regularization theorem qdelta-q estimate}. This is due to the lack of a Sturm-Liouville differential operator for the limited-aperture and multi-frequency cases (to the best of our knowledge). It may be possible to impose  the same explicit  assumption by combining the result of Theorem \ref{section SLP/regularization theorem qdelta-q estimate} and an extrapolation algorithm where  the partial data is extrapolated to approximate the full aperture data with a stability estimate.
 
 Now we apply Lemma \ref{theorem full aperture Ru-chiq estimate} to the limited-aperture data and multi-frequency data cases.
\begin{theorem}
Let $\underline{q}$ be the extension function that $\underline{q}=q$ in $\Omega$ and $\underline{q}=0$ outside $\Omega$ almost everywhere.
\begin{itemize}
\item
Suppose that  there exists $c_L$ such that   $\Omega \subset D_L=\{c_Lx/k: x\in L\}$ where $L$ is the interior of $\{\hat{\theta}-\hat{x}: \hat{x}, \hat{\theta} \in \mathbb{S}_L\}$ and $\mathbb{S}_L=\{x: x \in \mathbb{S}, \arg x \in [-\Theta,\Theta],  0<\Theta<\pi\}$. Let $\mathcal{R}_{\alpha}^L:L^2(D_L) \to L^2(D_L)$ be given by
\begin{equation*} 
\mathcal{R}^L_{\alpha} u = \sum_{(\frac{c_L}{k})^2 |\alpha^L_n(c_L)| >\alpha} \left\langle u , \frac{\psi^L_n(\cdot;c_L)}{\|\psi^L_n(\cdot;c_L)\|} \right\rangle_{D_L}   \frac{\psi^L_n(\cdot;c_L)}{\|\psi^L_n(\cdot;c_L)\|}.
\end{equation*}
Suppose that $\|  {   u_{b,L}^{\infty,\delta}} -  {   u_{b,L}^{\infty}}\|_{L^2(D_L)} \le \delta$. Let $\mathcal{K}^L$ be given by \eqref{limited-aperture ubL},  and suppose that $\underline{q} \in Range((\mathcal{K} ^{L*}\mathcal{K}^L )^{\sigma/2})$ for some $\sigma>0$ and $\|(\mathcal{K}^{L*}\mathcal{K}^L )^{-\sigma/2} \underline{q}\|_{L^2(D_L)} \le E_1$ for some constant $E_1$. Let $\alpha(\delta)=c_1 (\delta/E_1)^{1/(1+\sigma)}$ with some positive constant $c_1$; then it holds that
\begin{eqnarray*}
\|\mathcal{R}_{\alpha(\delta)}^L   {   u_{b,L}^{\infty,\delta}}  - \underline{q}\|_{L^2(D_L)} \le  \delta^{\sigma/(1+\sigma)} E_1^{1/(1+\sigma)}\big(1/c_1 + c_1^\sigma  \big).
\end{eqnarray*}
\item
Suppose that  there exists $c_M$ such that   $\Omega \subset D_M=\{c_Mx/K: x\in M\}$ where $M$  is the interior of   $\{a\hat{\theta} \pm a\hat{x}^*: \hat{\theta} \in \mathbb{S}, a \in (0,1)\}$. Let $\mathcal{R}_{\alpha}^M:L^2(D_M) \to L^2(D_M)$ be given by
\begin{equation*} 
\mathcal{R}^M_{\alpha} u = \sum_{(\frac{c_M}{k})^2 |\alpha^M_n(c_M)| >\alpha} \left\langle u , \frac{\psi^M_n(\cdot;c_M)}{\|\psi^M_n(\cdot;c_M)\|} \right\rangle_{D_M}   \frac{\psi^M_n(\cdot;c_M)}{\|\psi^M_n(\cdot;c_M)\|}.
\end{equation*}
Suppose that $\| {   u_{b,M}^{\infty,\delta}}  -  {   u_{b,M}^{\infty}}\|_{L^2(D_M)} \le \delta$. Let $\mathcal{K}^M$ be given by \eqref{multi-frequency data ubM}, and suppose that $\underline{q} \in Range((\mathcal{K} ^{M*}\mathcal{K}^M )^{\sigma/2})$ for some $\sigma>0$ and $\|(\mathcal{K}^{M*}\mathcal{K}^M )^{-\sigma/2} \underline{q}\|_{L^2(D_M)} \le E_2$ for some constant $E_2$. Let $\alpha(\delta)=c_2 (\delta/E_2)^{1/(1+\sigma)}$ with some positive constant $c_2$; then it holds that
\begin{eqnarray*}
\|\mathcal{R}_{\alpha(\delta)}^M   {   u_{b,M}^{\infty,\delta}} - \underline{q}\|_{L^2(D_M)} \le  \delta^{\sigma/(1+\sigma)} E_2^{1/(1+\sigma)}\big(1/c_2 + c_2^\sigma  \big).
\end{eqnarray*}
\end{itemize}
\end{theorem}
\begin{proof}
These conclusions are the direct consequences of Lemma \ref{theorem full aperture Ru-chiq estimate}.
\end{proof}

\section{Conclusion}
In this paper we propose  data-driven basis functions for reconstructing the medium contrast using Born data, including  the full aperture, limited-aperture, and multi-frequency partial data. The data-driven basis, which originated in the study of a Fourier integral operator,  allows us to establish a Picard criterion for reconstructing the contrast. Another salient feature is that such a data-driven basis remarkably extends analytically to $\mathbb{R}^2$, is  doubly orthogonal, and is complete in the class of band-limited functions. This yields that the reconstruction formula by Picard criterion can be understood from the viewpoint of data processing and analytic extrapolation. Another feature is that the data-driven basis for a disk is also a basis for a Sturm-Liouville differential operator. This Sturm-Liouville differential operator brings additional regularity estimates that lead to estimating the $L^2$ approximation error for a spectral cutoff approximation of functions in $H^s$. This  approximation theory allows us to obtain a spectral cutoff  regularization strategy with an explicit stability estimate for noisy data. 
In a broader context, the data-driven basis in this paper can also be learned via a Legendre-Galerkin neural network and our analysis indeed serves as a mathematical foundation towards relevant machine learning algorithms.
{ The extension of our work to $\mathbb{R}^3$ necessarily requires the study of the generalization of the prolate spheroidal wave functions in $\mathbb{R}^3$.} 
Perhaps the most exciting extension  is to investigate  a possible data-driven basis for  the full model in which case some non-linear transformation is most likely needed.  

\bibliographystyle{SIAM}

\end{document}